\documentclass[a4paper,12pt]{amsart}

\usepackage[margin=1.05in]{geometry}
\usepackage[hyperfootnotes=false]{hyperref}
\usepackage{amsmath,mathtools}
\usepackage{enumerate}
\usepackage{tikz}
\usepackage[all]{xy}
\usepackage{verbatim}
\usepackage{setspace}
\usepackage{enumitem}
\usepackage{mathdots}
\usepackage{xcolor}

\xyoption{rotate}

\newcommand\blfootnote[1]{
  \begingroup
  \renewcommand\thefootnote{}\footnote{#1}
  \addtocounter{footnote}{-1}
  \endgroup
}

\DeclareMathOperator{\tp}{top}
\DeclareMathOperator{\rad}{rad}
\DeclareMathOperator{\soc}{soc}

\DeclareMathOperator{\Mod*}{mod}

\DeclareMathOperator{\val}{val}

\theoremstyle{plain}
\newtheorem{thm}{Theorem}[section]
\newtheorem*{thm*}{Theorem}
\newtheorem{lem}[thm]{Lemma}
\newtheorem{cor}[thm]{Corollary}
\newtheorem*{cor*}{Corollary}
\newtheorem{prop}[thm]{Proposition}
\newtheorem*{prop*}{Proposition}

\theoremstyle{definition}
\newtheorem{defn}[thm]{Definition}
\newtheorem{defns}[thm]{Definitions}
\newtheorem{exam}[thm]{Example}
\newtheorem*{exam*}{Example}

\theoremstyle{remark}
\newtheorem{rem}[thm]{Remark}

\numberwithin{equation}{section}

\begin{document}
	\title{Auslander-Reiten Components of Symmetric Special Biserial Algebras}
	\author{Drew Duffield}
	\maketitle
	\blfootnote{\copyright2018. This manuscript version is made available under the CC-BY-NC-ND 4.0 license \href{http://creativecommons.org/licenses/by-nc-nd/4.0/}{http://creativecommons.org/licenses/by-nc-nd/4.0/}}
	
	\begin{abstract}
		We provide a combinatorial algorithm for constructing the stable Auslander-Reiten component containing a given indecomposable module of a symmetric special biserial algebra using only information from its underlying Brauer graph. We also show that the structure of the Auslander-Reiten quiver is closely related to the distinct Green walks of the Brauer graph and detail the relationship between the precise shape of the stable Auslander-Reiten components for domestic Brauer graph algebras and their underlying graph. Furthermore, we show that the specific component containing a given simple or indecomposable projective module for any Brauer graph algebra is determined by the edge in the Brauer graph associated to the module.
	\end{abstract}
	
	\section{Introduction}
		Within the study of the representation theory of finite dimensional algebras, one of the primary aims is to understand the indecomposable modules of the algebra along with the morphisms between them. The Auslander-Reiten quiver of an algebra is a means of presenting this information.

	The topic of this paper is on the Auslander-Reiten quiver of a class of algebras known as Brauer graph algebras, which coincides with the class of symmetric special biserial algebras (\cite{Roggenkamp}, \cite{trivialGentle}). They are finite dimensional algebras constructed from a decorated ribbon graph. The study of these algebras originates from the work of Richard Brauer on the modular representation theory of finite groups. Brauer graph algebras have since been studied extensively by various authors (see for example in \cite{SchrollGroup}, \cite{Marsh}, \cite{Rickard}, \cite{Roggenkamp}).

	It is known that Brauer graph algebras are of tame representation type (\cite{Alperin}, \cite{Benson}, \cite{WaldWasch}) and those of finite representation type are precisely the Brauer tree algebras. The underlying Brauer graph of a domestic symmetric special biserial algebra has been described in \cite{bocianSkow}. It is also known that the indecomposable non-projective modules of a Brauer graph algebra are given by either string modules or band modules (\cite{WaldWasch}). The irreducible morphisms between indecomposable modules are then given by adding or deleting hooks and cohooks to strings (\cite{butlerRingel}, \cite{Biserial}).

	One of the interesting properties of Brauer graph algebras is that it is possible to read off some of the representation theory of the algebra from it's underlying Brauer graph. For example, a useful tool in representation theory is the projective resolution of a module. However, projective resolutions in algebras are difficult to calculate in general. For Brauer graph algebras, one can avoid such calculations and easily read off the projective resolutions of certain modules from the Brauer graph. These are given by the Green walks of the Brauer graph, which were first described in detail in \cite{GreenWalk} for Brauer trees and are shown to hold more generally in \cite{Roggenkamp}.

	The aim of this paper is to study what information about the Auslander-Reiten quiver of a Brauer graph algebra we can read off from its underlying Brauer graph. There has already been extensive work on the Auslander-Reiten quiver of Brauer tree algebras. For example, a complete description of the Auslander-Reiten quiver of Brauer tree algebras, has been given in \cite{GroupsWOGroups}. In \cite{Chinburg}, the location of the modules in the stable Auslander-Reiten quiver of a Brauer tree algebra has been described in terms of walks in the Brauer tree. However, the descriptions in both \cite{Chinburg} and \cite{GroupsWOGroups} do not address the case where the algebra is of infinite representation type, and thus, is associated to a Brauer graph that is not a Brauer tree.
	
	In Section~\ref{Algorithm}, we provide an algorithm for constructing the stable Auslander-Reiten component of a given string module of a Brauer graph algebra using only information from its underlying Brauer graph. This algorithm is of particular importance because it allows us to describe the string combinatorics of the algebra in terms of the Brauer graph. This algorithm thus allows us to prove several results later in the paper, which relate the number and shape of the components of the Auslander-Reiten quiver of the algebra to its underling Brauer graph.

	In \cite{ErdmannAR}, the Auslander-Reiten components of self-injective special biserial algebras have been described. In particular, any Brauer graph algebra has a finite number of exceptional tubes in the stable Auslander-Reiten quiver. In Section~\ref{Components}, we show that the rank and the total number of these tubes is closely related to the distinct Green walks of the Brauer graph. Specifically, we prove the following.

	\begin{thm}
		Let $A$ be a representation-infinite Brauer graph algebra with Brauer graph $G$ and let $_s\Gamma_A$ be its stable Auslander-Reiten quiver. Then
	 	\begin{enumerate}[label=(\alph*)]
	 		\item there is a bijective correspondence between the exceptional tubes in $_s\Gamma_A$ and the distinct double-stepped Green walks of $G$; and
	 		\item the rank of an exceptional tube is given by the length of the corresponding double-stepped Green walk of $G$.
	 	\end{enumerate}
	\end{thm}
	
	It is shown in \cite{ErdmannAR} that the Auslander-Reiten components of Brauer graph algebras are strongly related to the growth type of the algebra. The algebra, for example, contains a Euclidean component of the form $\mathbb{Z}\tilde{A}_{p,q}$ if and only if the algebra is domestic and of infinite representation type. We use a simple application of the above theorem to show the following results regarding the values of $p$ and $q$.
	\begin{thm}
		Let $A$ be a Brauer graph algebra constructed from a graph $G$ of $n$ edges and suppose $_s\Gamma_A$ has a $\mathbb{Z}\tilde{A}_{p,q}$ component.
		\begin{enumerate}[label=(\alph*)]
			\item If $A$ is 1-domestic, then $p+q=2n$.
			\item If $A$ is 2-domestic, then $p+q=n$.
		\end{enumerate}
		Furthermore, if $G$ is a tree, then $p=q=n$.
	\end{thm}
	
	\begin{cor} \label{TubeType}
		 For a domestic Brauer graph algebra, if the Brauer graph contains a unique (simple) cycle of length $l$ and there are $n_1$ additional edges on the inside of the cycle and $n_2$ additional edges along the outside, then the $\mathbb{Z}\tilde{A}_{p,q}$ components are given by
		\begin{equation*}
			p=
			\begin{cases}
				l+2n_1 & l \text{ odd,} \\
				\frac{l}{2}+n_1 & l \text{ even,}
			\end{cases}
			\quad \text{and} \quad
			q = 
			\begin{cases}
				l+2n_2 & l \text{ odd,} \\
				\frac{l}{2}+n_2 & l \text{ even.}
			\end{cases}
		\end{equation*}
	\end{cor}
	
	In the later parts of Section~\ref{Components}, we use the algorithm presented in Section~\ref{Algorithm} to prove results which show how one can determine the specific component containing a given simple or indecomposable projective module from its associated edge in the Brauer graph $G$.
	
	We first divide the edges of $G$ into two distinct classes, which we call \emph{exceptional} and \emph{non-exceptional edges}. The exceptional edges of a Brauer graph belong to a special class of subtrees of the graph, which we refer to as the \emph{exceptional subtrees} of the Brauer graph. Intuitively, one can think of these subtrees as belonging to parts of the algebra that behave locally as a Brauer tree algebra. We then prove the following regarding the exceptional edges of a Brauer graph.
	\begin{thm}
		Let $A$ be a representation-infinite Brauer graph algebra associated to a graph $G$ and let $x$ be an edge in $G$. Then the simple module and the radical of the indecomposable projective module associated to $x$ belong to exceptional tubes of $_s\Gamma_A$ if and only if $x$ is an exceptional edge.
	\end{thm}
	
	We also determine when a simple module and the radical of an indecomposable projective module belong to the same Auslander-Reiten component. For exceptional edges, we have the following.
	
	\begin{cor}
		Given an exceptional edge $x$ in a Brauer graph, the simple module and the radical of the indecomposable projective module associated to $x$ belong to the same exceptional tube if and only if we walk over both vertices connected to $x$ in the same double-stepped Green walk.
	\end{cor}
	
	We then finally prove a similar result for non-exceptional edges, 
	\begin{thm} 
		Let $A$ be a representation-infinite Brauer graph algebra associated to a Brauer graph $G$ and let $\mathfrak{e}_{v}$ denote the multiplicity of a vertex $v$ in $G$. Let $x$ and $y$ be non-exceptional edges of $G$. Then the simple module associated $x$ is in the same component of $_s\Gamma_A$ as the radical of the indecomposable projective module associated to $y$ if and only if either $A$ is 1-domestic or there exists a path 
		\begin{equation*}
			p: \xymatrix@1{u_0 \ar@{-}[r]^{x_1=x} & v_1 \ar@{-}[r]^{x_2} & v_2 \ar@{-}[r] & \cdots \ar@{-}[r] & v_{n-2} \ar@{-}[r]^-{x_{n-1}} & v_{n-1} \ar@{-}[r]^-{x_n=y} & u_1}
		\end{equation*}
		of even length in $G$ consisting of non-exceptional edges such that
		\begin{enumerate}[label=(\roman*)]
			\item every edge $x_i$ is not a loop;
			\item $\mathfrak{e}_{v_i}=1$ if $x_i \neq x_{i+1}$ and $\mathfrak{e}_{v_i}=2$ if $x_i = x_{i+1}$;
			\item $x_i$ and $x_{i+1}$ are the only non-exceptional edges incident to $v_i$ in $G$.
		\end{enumerate}
	\end{thm}
	
	Whilst this paper assumes the algebras we are dealing with are symmetric special biserial, many of the results in this paper should carry over to the weakly symmetric case (and hence, should work for the quantised Brauer graph algebras defined, for example, in \cite{SchrollGroup}). Indeed, the indecomposable modules of a weakly symmetric special biserial algebra are again given by string and band modules (see for example \cite{ErdmannString}).
	
	This paper forms part of the author's PhD thesis, and has been \href{https://doi.org/10.1016/j.jalgebra.2018.03.040}{accepted for publication} in The Journal of Algebra.
		
	\section*{Acknowledgements}
	I would like to deliver a special thank you to my supervisor, Sibylle Schroll, for taking the time to proofread my paper. I would also like to thank the EPSRC (award reference: 1366175) for funding my PhD, which has allowed me to undertake this research.
	
	\section{Preliminaries} \label{Prelim}
	Before we proceed with the results of the paper, we will briefly outline several definitions from various sources and establish a general set of notations. Throughout, we let $K$ be an algebraically closed field and $Q$ a finite connected quiver with vertex set $Q_0$ and arrow set $Q_1$. We let $I$ be an admissible ideal of the path algebra $KQ$ such that $KQ/I$ is finite dimensional. We denote by $\Mod* A$ the category of finitely generated $A$-modules. All modules considered are right modules and thus, we typically read paths in a quiver from left to right.
	
	\subsection{Special biserial algebras}
	We recall that a module $M$ is called \emph{uniserial} if $\rad^iM/\rad^{i+1}M$ is simple or zero for all $i$; and that an indecomposable projective module $P$ is called \emph{biserial} if $\rad P$ is a sum of at most two uniserial submodules whose intersection is simple or zero.
	
	An algebra is called \emph{special biserial} if it is Morita equivalent to an algebra of the form $KQ/I$ such that
	\begin{enumerate}[label=(SB\arabic*)]
		\item each vertex of $Q$ is the source of at most two arrows and target of at most two arrows in $Q$; and
		\item for each arrow $\alpha$ in $Q$, there exists at most one arrow $\beta$ such that $\beta\alpha \not\in I$ and at most one arrow $\gamma$ such that $\alpha\gamma \not\in I$.
	\end{enumerate}
	
	\subsection{Brauer graph algebras} \label{BGAPrelim}
	Let $G$ be finite, connected, undirected, non-empty graph (with loops and multiple edges permitted). We say $G$ is a \emph{Brauer graph} if $G$ is equipped with a cyclic ordering of the edges adjacent to each vertex, and if each vertex is equipped with a strictly positive integer called the \emph{multiplicity} of the vertex. Loops incident to a vertex appear twice in the cyclic ordering at the vertex. We realise a Brauer graph as a local embedding of the edges adjacent to each vertex into the oriented plane. We use an anticlockwise cyclic ordering throughout. We typically denote the multiplicity of a vertex $v$ in $G$ by $\mathfrak{e}_v$. We call a Brauer graph a \emph{Brauer tree} if $G$ is a tree and at most one vertex $v$ in $G$ has multiplicity $\mathfrak{e}_v>1$.
	
	Let $v$ be a vertex of $G$. We define the \emph{valency} of $v$ to be the positive integer $\val(v)=n$, where $n$ is the total number of edges incident to $v$. A loop incident to a vertex is counted as two edges. We call $v$ a \emph{truncated vertex} of $G$ if $\val(v)=1$ and $\mathfrak{e}_v=1$ and a \emph{non-truncated vertex} of $G$ otherwise. We call an edge $x$ a \emph{truncated edge} of $G$ if $x$ is incident to a truncated vertex of $G$, and we call $x$ a \emph{non-truncated edge} otherwise.
	
	Suppose $x_1, x_2$ are edges incident to a vertex $v$ in a Brauer graph $G$. We say $x_2$ is a \emph{successor} to the edge $x_1$ at $v$ if $x_2$ directly follows $x_1$ in the cyclic ordering at $v$. From this we obtain a sequence $x_1, x_2, \ldots, x_{\val(v)}$, where each $x_{i+1}$ is the successor to $x_i$ at $v$ and $x_1$ is the successor to $x_{\val(v)}$ at $v$. We call this the \emph{successor sequence} of $x_1$ at $v$. By our notation, this corresponds to the edges in $G$ given by locally walking anticlockwise at $v$ from the edge $x_1$ in the oriented plane. We consider a truncated edge $x$ to be the successor to itself at its truncated vertex. If $x$ is a loop in the Brauer graph at a vertex $v$, then $x$ occurs twice in any successor sequence at $v$.
	
	Similarly, suppose $y_1, y_2$ are edges incident to a vertex $v$ in $G$. We say $y_2$ is a \emph{predecessor} to $y_1$ at $v$ if $y_1$ directly follows $y_2$ in the cyclic ordering at $v$. We also obtain a sequence $y_{\val(v)}, \ldots, y_2,  y_1$, where each $y_{i+1}$ is the predecessor to $y_i$ at $v$ and $y_1$ is the predecessor to $y_{\val(v)}$ at $v$. We call this the \emph{predecessor sequence} of $y_1$ at $v$. Equivalently, this sequence is the successor sequence run in reverse and corresponds to the edges in $G$ given by walking clockwise at $v$ from $y_1$. We consider a truncated edge $y$ to be the predecessor to itself at its truncated vertex. If $y$ is a loop at a vertex $v$, then $y$ occurs twice in any predecessor sequence at $v$.
	
	Let $x_1$ be an edge incident to a vertex $v$ in $G$. To each term $x_i$ ($i \in \{1, \ldots, \val(v)\}$) in the successor sequence of $x_1$ at $v$, we identify a \emph{half-edge} $x_i^v$. In the case where there exist edges $x_i=x_j$ for some $1 \leq i<j\leq \val(v)$ (that is, where $x_i=x_j$ is a loop), we define the half-edges $x_i^v$ and $x_j^v$ such that $x_i^v\neq x_j^v$. Consequently, for each edge $\xymatrix@1{u \ar@{-}[r]^{x} & v}$ in the Brauer graph, we associate precisely two half-edges $x^{u}$ and $x^{v}$. We define an involution operation $\overline{\,\cdot\,}$ on a half-edge by $\overline{x^{u}} = x^{v}$ and $\overline{x^{v}} = x^{u}$. If $u=v$ (that is, $x$ is a loop) then we will often distinguish the two half-edges associated to $x$ by $x^{u}$ and $\overline{x^{u}}$.
	
	Given a successor sequence $x_1, \ldots, x_{\val(v)}$ of an edge $x_1$ at a vertex $v$, we say each half-edge $x_{i+1}^v$ is the \emph{successor} to $x_i^v$. From this, we can define the \emph{successor sequence of the half-edge} $x_1^v$ to be the sequence $x_1^v, x_2^v, \ldots, x_{\val(v)}^v$. Similarly given a predecessor sequence $y_{\val(v)}, \ldots, y_2, y_1$ of an edge $y_1$ at a vertex $v$, we say each half-edge $y_{i+1}^v$ is the \emph{predecessor} to $y_i^v$. We define the \emph{predecessor sequence of the half-edge} $y_1^v$ to be the sequence $y_{\val(v)}^v, \ldots, y_2^v, y_1^v$.
	
	Throughout the paper, we make use of graph theoretic paths and cycles, which we define formally here.
	\begin{defn} \label{PathDefn}
		Let $G$ be a Brauer graph. We define a \emph{path} of length $n$ in $G$ to be a sequence of vertices and half-edges 
		\begin{equation*}
			(v_0, x_1^{v_0}, x_1^{v_1}, v_1, x_2^{v_1}, x_2^{v_2}, v_2, \ldots, x_n^{v_{n-1}}, x_n^{v_n}, v_n),
		\end{equation*}
		where each $x_i$ is connected to the vertices $v_{i-1}$ and $v_i$, and each $x_i^{v_i}=\overline{x_i^{v_{i-1}}}$.
	\end{defn}
	It is sufficient to describe a path of the form in Definition~\ref{PathDefn} as a sequence of half-edges. However we often wish to utilise the incident vertices and both half-edges associated to an edge when using paths, and so we include this additional data. For readability purposes, we shall write a path of the form in Definition~\ref{PathDefn} as 
	\begin{equation*}
		\xymatrix@1{v_0 \ar@{-}[r]^-{x_1} & v_1 \ar@{-}[r]^-{x_2} & \cdots \ar@{-}[r]^-{x_{n-1}} & v_{n-1} \ar@{-}[r]^-{x_n} & v_n}
	\end{equation*}
	and if necessary, clarify the order of the half-edges in the event of potential ambiguities resulting from loops and multiple edges.
	\begin{defns}
		Let $G$ be a Brauer graph and $p$ be a path in $G$.
		\begin{enumerate}[label=(\roman*)]
			\item We call $p$ a \emph{simple path} if it is non-crossing at vertices. That is, $v_i \neq v_j$ for all $i \neq j$.
			\item We call $p$ a \emph{cycle} of $G$ if the starting and ending vertices are the same. A cycle is said to be \emph{simple} if it is a simple path (except for the starting and ending vertices).
		\end{enumerate}
	\end{defns}
	
	Following \cite{GreenWalk}, we make the following important definitions.
	\begin{defns}
		Let $G$ be a Brauer graph.
		\begin{enumerate}[label=(\roman*)]
	 		\item A \emph{Green walk} of $G$ from an edge $x_0$ via a vertex $v_0$ is a (periodic) sequence $(x_j^{v_j})_{j\in\mathbb{Z}_{\geq 0}}$ of half-edges such that $x_i$ is connected to $x_{i+1}$ via the vertex $v_i$ and $\overline{x_{i+1}^{v_{i+1}}}$ is the successor to $x_i^{v_i}$. By a \emph{clockwise Green walk} from $x_0$ via $v_0$, we mean a similar sequence $(x_j^{v_j})_{j\in\mathbb{Z}_{\geq 0}}$ of half-edges such that each $\overline{x_{i+1}^{v_{i+1}}}$ is the predecessor to $x_i^{v_i}$.
	 		\item By a \emph{double-stepped Green walk} of $G$, we mean a subsequence $(x_{2k}^{v_{2k}})_{k\in\mathbb{Z}_{\geq 0}}$ of a Green walk $(x_j^{v_j})_{j\in\mathbb{Z}_{\geq 0}}$.
	 		\item We say a Green walk (anticlockwise, clockwise and/or double-stepped) is \emph{of length} $l$ if it is of period $l$ -- that is, $l$ is the least integer such that $x_i^{v_i} = x_{l+i}^{v_{l+i}}$ for all $i$.
	 		\item We say two Green walks $(x_j^{v_j})_{j\in\mathbb{Z}_{\geq 0}}$ and $(y_j^{u_j})_{j\in\mathbb{Z}_{\geq 0}}$ are \emph{distinct} if there exists no integer $k$ such that $x_{i+k}^{v_{i+k}}=y_i^{u_i}$ for all $i$.
		\end{enumerate}
	\end{defns}
	
	Let $G$ be a Brauer graph $G$ and let $X$ be the set of all edges in $G$. We construct an algebra $A$ associated to $G$ as follows. If $G$ is the graph $\xymatrix{u \ar@{-}[r] & v}$ with $\mathfrak{e}_u=\mathfrak{e}_v=1$ then let $Q$ be the quiver with one vertex and a single loop $\alpha$. We further define a single relation on $Q$ by $\alpha^2$. Otherwise, we define $Q$ such that the vertices are in bijective correspondence with the edges of $G$. Let $\sigma:X \rightarrow Q_0$ be such a bijection. If $x_2$ is the direct successor to $x_1$ at a non-truncated vertex in $G$, then there exists an arrow $\sigma(x_1) \rightarrow \sigma(x_2)$ in $Q$. If $x$ is connected to a vertex $v$ such that $\val(v)=1$ and $\mathfrak{e}_v>1$, then there exists a loop in $Q$ at $\sigma(x)$. If $\mathfrak{e}_v=1$, then no such loop exists. Note that it follows from this construction that each non-truncated vertex in $G$ generates a cycle of arrows in $Q$, and no two such cycles share a common arrow. We denote by $\mathfrak{C}_v$ the cycle of $Q$ (up to cyclic rotation) generated by a non-truncated vertex $v$ and by $\mathfrak{C}_{v,\alpha}$ the particular rotation of $\mathfrak{C}_v$ such that the first arrow in the cycle is $\alpha$. 
	
	The next step in the construction is to define a set of relations $\rho$ on $Q$, which will be precisely the relations of the algebra $A$ associated to $G$. If $x$ is a truncated edge of $G$ and $\mathfrak{C}_{v,\gamma_1}=\gamma_1\ldots\gamma_n$ is the cycle generated by the non-truncated vertex $v$ connected to $x$, where $\gamma_1$ is of source $x$, then $(\mathfrak{C}_{v,\gamma_1})^{\mathfrak{e}_v}\gamma_1\in\rho$. If $\xymatrix@1{u \ar@{-}[r]^x & v}$ is a non-truncated edge, and $\mathfrak{C}_{u,\gamma}, \mathfrak{C}_{v,\delta}$ are the two distinct rotations of cycles of source $x$ at the respective vertices $u,v$ (where $u=v$, in the case of a loop), then $(\mathfrak{C}_{u,\gamma})^{\mathfrak{e}_u}-(\mathfrak{C}_{v,\delta})^{\mathfrak{e}_v}\in\rho$. Finally, if $\alpha\beta$ is a path of length 2 in $Q$ and is not a subpath of any cycle generated by any vertex of $G$, then $\alpha\beta\in\rho$. The algebra $A=KQ/I$, where $I$ is the ideal generated by $\rho$, is called the \emph{Brauer graph algebra} associated to $G$. Thus, the relations in $\rho$ define precisely the relations for the Brauer graph algebra associated to $G$.
	
	Since edges in a Brauer graph $G$ are in bijective correspondence with the vertices in the quiver of the associated Brauer graph algebra, we will often consider an edge $x$ as both an edge in $G$ and a vertex in $Q$. That is, by an abuse of notation, we will often write an element $\sigma(x) \in Q_0$ as $x$. We label the simple and indecomposable projective(-injective) modules in the Brauer graph algebra corresponding to an edge $x$ by $S(x)$ and $P(x)$ respectively. We similarly denote the projective cover and injective envelope of an $A$-module $M$ by $P(M)$ and $I(M)$ respectively.
	
	\subsection{Strings, bands and string modules}
	We follow the definitions from \cite{butlerRingel}, with some slight alterations to notation. Let $A=KQ/I$ be a Brauer graph algebra associated to a Brauer graph $G$. For each arrow $\alpha \in Q_1$, we denote by $s(\alpha)$ the source of $\alpha$ and $e(\alpha)$ the target of $\alpha$. We denote by $\alpha^{-1}$ the formal inverse of $\alpha$. That is, the symbolic arrow given by $s(\alpha^{-1})=e(\alpha)$ and $e(\alpha^{-1})=s(\alpha)$. By $Q_1^{-1}$, we denote the set of all formal inverses of arrows in $Q_1$.
	
	Since the vertices of $Q$ correspond to the edges of $G$, the source and target of any arrow in $Q_1$ or formal inverse in $Q^{-1}_1$ can be considered as edges of $G$. On the other hand, a half-edge of $G$ can be considered as the source or target of an arrow in $Q$ in the following way. Let $v$ be a non-truncated vertex of $G$ and let $x_1$ be an edge incident to $v$. Let $x_1, x_2,\ldots, x_{\val(v)}$ be the successor sequence of $x_1$ at $v$. Then there exists an arrow $\xymatrix@1{x_i \ar[r]^{\alpha_i} & x_{i+1}}$ in $Q$ for each $1 \leq i \leq \val(v)$. However, the terms $x_i$ and $x_{i+1}$ in the successor sequence correspond to distinct half-edges $x_{i}^v$ and $x_{i+1}^v$ respectively. We therefore define the \emph{source half-edge} of $\alpha_i$ to be $x_i^v$, which we denote by $\widehat{s}(\alpha_i)$. We define the \emph{target half-edge} of $\alpha_i$ to be $x_{i+1}^v$, which we denote by $\widehat{e}(\alpha)$. For a formal inverse $\alpha^{-1}$ of an arrow $\alpha$, we define $\widehat{s}(\alpha^{-1})=\widehat{e}(\alpha)$ and $\widehat{e}(\alpha^{-1})=\widehat{s}(\alpha)$. We note that the definitions of the functions $\widehat{s}$ and $\widehat{e}$ are not present in \cite{butlerRingel}, since Brauer graphs do not appear there.
	
	We call a word $w=\alpha_1\ldots\alpha_n$ with $\alpha_i \in Q_1 \cup Q_1^{-1}$ a \emph{string} if $\alpha_{i+1} \neq \alpha_i^{-1}$, $s(\alpha_{i+1})=e(\alpha_i)$ and no subword nor its inverse occurs in some relation in $\rho$ (where $\rho$ is as defined in Subsection~\ref{BGAPrelim}). We define the \emph{length} of $w$ to be the positive integer $n$, which we denote by $|w|$. We say $w$ is a \emph{direct string} (resp. an \emph{inverse string}) if $\alpha_i \in Q_1$ (resp. $\alpha_i \in Q_1^{-1}$) for all $i$. We define $s(w)=s(\alpha_1)$ and $e(w)=e(\alpha_n)$. A \emph{band} in $A$ is defined to be a cyclic string $b$ such that each power $b^m$ is a string, but $b$ is not a proper power of any string. One can therefore view a band as a cyclic walk along the arrows and formal inverses of $Q$.
	
	We call a string of length zero a \emph{zero string}. To each vertex $x \in Q_0$, we associate precisely one zero string, which we will sometimes denote by the stationary path $\varepsilon_x$ at $x$. If $w=\varepsilon_x$ is a zero string, then we let $s(w)=x=e(w)$. However, the functions $\widehat{s}$ and $\widehat{e}$ are not defined for zero strings. A zero string is defined to be both direct and inverse.
	
	We say a string $w$ \emph{starts on a peak} (resp. \emph{ends on a peak}) if no arrow $\alpha$ exists such that $\alpha w$ (resp. $w \alpha^{-1}$) is a string in $A$. We say $w$ \emph{starts in a deep} (resp. \emph{ends in a deep}) if no arrow $\alpha$ exists such that $\alpha^{-1} w$ (resp. $w \alpha$) is a string in $A$.
	
	For any arrow $\alpha \in Q_1$, let $u_\alpha$ and $v_\alpha$ be the unique inverse strings such that $u_\alpha\alpha$ is a string that starts in a deep and $\alpha v_\alpha$ is a string that ends on a peak. Note that this means $\alpha^{-1} u_{\alpha}^{-1}$ ends in a deep and $v_\alpha^{-1} \alpha^{-1}$ starts on a peak. Suppose $w$ is a string that does not end on a peak. Then there exists an arrow $\beta \in Q_1$ such that $w\beta^{-1}$ is a string. We say the string $w_h = w\beta^{-1}u_\beta^{-1}$ is obtained from $w$ by \emph{adding a hook to the end of} $w$. Similarly, suppose $w$ is a string that does not start on a peak. Then there exists an arrow $\gamma \in Q_1$ such that $\gamma w$ is a string. We say the string $_hw = u_\gamma\gamma w$ is obtained from $w$ by \emph{adding a hook to the start of} $w$. If instead we have $w=w_{-h}\alpha^{-1}u_\alpha^{-1}$ for some substring $w_{-h}$, we say $w_{-h}$ is obtained from $w$ by \emph{deleting a hook from the end of} $w$. Similarly, if $w=u_\alpha\alpha(_{-h}w)$, we say $_{-h}w$ is obtained from $w$ by \emph{deleting a hook from the start of} $w$.
	
	Now suppose that a string $w$ does not end in a deep. Then there exists an arrow $\beta \in Q_1$ such that $w\beta$ is a string. We say the string $w_c=w\beta v_\beta$ is obtained from $w$ by \emph{adding a cohook to the end of} $w$. Similarly, suppose $w$ is a string that does not start in a deep. Then there exists an arrow $\gamma \in Q_1$ such that $\gamma^{-1}w$ is a string. We say the string $_cw = v^{-1}_\gamma\gamma^{-1} w$ is obtained from $w$ by \emph{adding a cohook to the start of} $w$. If instead we have $w=w_{-c}\alpha v_\alpha$ for some substring $w_{-c}$, we say $w_{-c}$ is obtained from $w$ by \emph{deleting a cohook from the end of} $w$. Similarly, if $w=v^{-1}_\alpha\alpha^{-1}(_{-c}w)$, we say $_{-c}w$ is obtained from $w$ by \emph{deleting a cohook from the start of} $w$.
	
	In the case where $w$ is a non-zero string, the strings $w_h$, $w_c$, $_hw$ and $_cw$ are unique by property (SB2) in the definition of a special biserial algebra. The case where $w$ is a zero string is addressed in Subsection~\ref{HooksCohooksSec} and Subsection~\ref{LocSimpleProj}. We caution the reader that whilst the definitions for adding hooks and deleting cohooks is the same as in \cite{butlerRingel} our notation for $w_h$, $w_c$, $_hw$ and $_cw$ is different to \cite{butlerRingel}. This is because \cite{butlerRingel} only considers Auslander-Reiten sequences starting in a given string module, whereas we would like the flexibility to consider the Auslander-Reiten sequences ending in a given string module as well as those starting in a given string module.
	
	Let $w$ be an arbitrary string. We recall from \cite{butlerRingel} that $w$ determines a \emph{string module} $M(w) \in \Mod* A$. It follows from \cite{butlerRingel} that $M(w_1) \cong M(w_2)$ if and only if $w_1 =w_2$ or $w_1=w_2^{-1}$. Moreover, if $w$ is a zero string $\varepsilon_x$, then $M(w)=S(x)$. We refer the reader to \cite{butlerRingel} for the precise details of the construction of $M(w)$. There is also the notion of a band module. Each distinct band induces an infinite family of band modules. However, band modules are not the focus of this paper, and so we also refer the reader to \cite{butlerRingel} for the details of their construction.

	\section{A Constructive Algorithm} \label{Algorithm}
	It is already known that the non-projective indecomposable modules of a Brauer graph algebra $A$ are given by string and band modules. Given a string module $M$ in a Brauer graph algebra $A=KQ/I$ constructed from a graph $G$, we wish to be able to read off the (stable) Auslander-Reiten component containing $M$ from $G$.  Since the irreducible morphisms between string modules are given by adding or deleting hooks and cohooks to strings, our algorithm will need to take hooks and cohooks into account.
	\subsection{Presenting strings on a Brauer graph}
	 To achieve the aim of this section, we must first describe a method for presenting strings on Brauer graphs. This technique appears to be well known within the subject area. However, it forms an essential part of this paper and thus, we will outline the process here.
	 
	 Let $w=\alpha_1\ldots\alpha_n$ be a (not necessarily direct) string in $A$. Since each vertex in $G$ generates a cycle in $Q$ and no two such cycles share a common arrow in $Q$, we can associate to each arrow (or formal inverse) $\alpha_i$ in the string a vertex in $G$. Suppose $\alpha_i \in Q_1$ is an arrow belonging to the cycle generated by a vertex $v$ and $s(\alpha_i)=x$ and $e(\alpha_i)=y$. Since $x$ and $y$ are edges in the graph, we realise this arrow on $G$ as an anticlockwise arrow around the vertex $v$ of source $x$ and target $y$. Similarly, if we instead have $\alpha_i \in Q^{-1}_1$ such that the arrow $\alpha^{-1}_i$ belongs to the cycle generated by $v$, then we realise this as moving clockwise in $G$ around $v$ from $s(\alpha_i)$ to $e(\alpha_i)$. An example is given in Figure~\ref{StringEg}. Note that if we wish to invert a string, we simply flip the direction of the arrows and formal inverses drawn on the Brauer graph.
	 
	  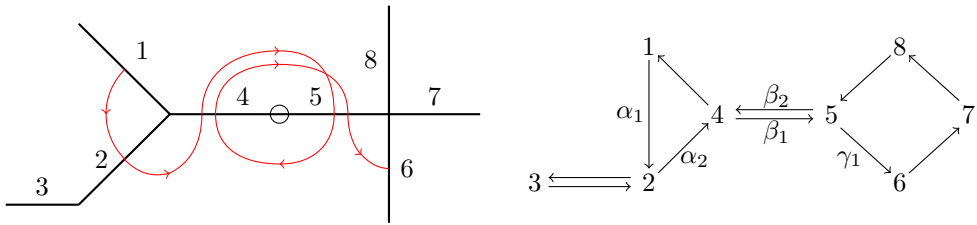
\begin{figure}[b]
	  	\centering
	 	\begin{tikzpicture}[scale=0.6]
			\draw[thick] (-1.6,-2) -- (0, -2);
			\draw[thick] (0,2) -- (2,0) -- (0,-2);
			\draw[thick] (2,0) -- (4.4,0) -- (6.8,0) -- (6.8,2.4);
			\draw[thick] (8.8,0) -- (6.8,0) -- (6.8,-2.4);
			\draw (4.4,0) ellipse (0.2 and 0.2);
	
			\footnotesize
			\draw (1.4,1.4) node { $1$};
			\draw (0.5,-1) node {$2$};
			\draw (-0.8,-1.6) node {$3$};
			\draw (3.6,0.4) node {$4$};
			\draw (5.2,0.4) node {$5$};
			\draw (7.2,-1.2) node {$6$};
			\draw (7.8,0.4) node {$7$};
			\draw (6.4,1.2) node {$8$};

			\draw (12.5,1.5) node {$1$};
			\draw (12.5,-1.5) node {$2$};
			\draw (10,-1.5) node {$3$};
			\draw (14,0) node {$4$};
			\draw (16.5,0) node {$5$};
			\draw (18,-1.5) node {$6$};
			\draw (19.5,0) node {$7$};
			\draw (18,1.5) node {$8$};
		
			\draw (12.1,0) node {$\alpha_1$};
			\draw (13.5,-1) node {$\alpha_2$};
			\draw (15.3,-0.4) node {$\beta_1$};
			\draw (15.3,0.4) node {$\beta_2$};
			\draw (16.9,-1) node {$\gamma_1$};
	
			\draw [->](10.3,-1.6) -- (12.1,-1.6);
			\draw [->](12.1,-1.4) -- (10.3,-1.4);
			\draw [->](12.5,1.2) -- (12.5,-1.2);
			\draw [->](12.7,-1.3) -- (13.8,-0.2);
			\draw [->](13.8,0.2) -- (12.7,1.3);
			\draw [->](14.4,-0.1) -- (16.1,-0.1);
			\draw [->](16.1,0.1) -- (14.4,0.1);
			\draw [->](16.7,-0.3) -- (17.8,-1.3);
			\draw [->](18.2,-1.3) -- (19.3,-0.3);
			\draw [->](17.8,1.3) -- (16.7,0.3);
			\draw [->](19.3,0.3) -- (18.2,1.3);

			\draw [red][->](1.0101,0.9899) arc (135:180:1.3999);
			\draw [red](0.6,0) arc (180:225:1.4);
			\draw [red][->](1,-1) .. controls (1.3,-1.3) and (1.6,-1.4) .. (2,-1.3);
			\draw [red](2,-1.3) .. controls (2.4,-1.1) and (2.7,-0.7) .. (2.7,0);
			\draw [red][->](2.7,0) .. controls (2.7,0.9) and (3.4,1.4) .. (4.4,1.4);
			\draw [red][->](4.4,1.4) .. controls (5.4,1.4) and (5.6,0.7) .. (5.6,0) .. controls (5.6,-0.7) and (5.1,-1.1) .. (4.4,-1.1);
			\draw [red][->](4.4,-1.1) .. controls (3.8,-1.1) and (3,-0.9) .. (3,0) .. controls (3,0.9) and (3.8,1.1) .. (4.4,1.1);
			\draw [red][->](4.4,1.1) .. controls (5.3,1.1) and (5.9,0.8) .. (5.9,0) .. controls (5.9,-0.5) and (6.1,-0.8) .. (6.2,-0.9);
			\draw [red](6.2,-0.9) .. controls (6.3,-1) and (6.5,-1.2) .. (6.8,-1.2);
		\end{tikzpicture}
	 	\caption{A Brauer tree $T$ (left) and its corresponding quiver (right) with the string $w=\alpha_1\alpha_2\beta^{-1}_2\beta^{-1}_1\beta^{-1}_2\gamma_1$ presented on the Brauer tree in red. The circled vertex in $T$ has a multiplicity of two, and all other vertices have a multiplicity of one.} \label{StringEg}
	 \end{figure}
	 
	 We often wish to perform this procedure in reverse. That is, we may wish to draw a sequence (or path) of connected `arrows' through the edges of $G$ and interpret this as a string in the algebra. To do this, we must describe which of these paths give valid strings. First recall that for a string $w=\alpha_1\ldots\alpha_n$, we require $\alpha_{i+1} \neq \alpha^{-1}_i$. Thus, we cannot draw an arrow anticlockwise around a vertex $v$ from an edge $x$ to an edge $y$ followed by a clockwise formal inverse around $v$ from $y$ to $x$ (and vice versa). Furthermore, if $v$ is a truncated vertex, then $v$ generates no arrows in $Q$ and thus, we cannot draw any arrows around $v$ in the graph.
	 
	 A string must also avoid the relations in $I$, so if we have edges $\xymatrix@1{\ar@{-}[r]^x & v \ar@{-}[r]^-y & v' \ar@{-}[r]^-z &}$, we cannot draw an anticlockwise (resp. clockwise) arrow (resp. formal inverse) around $v$ from $x$ to $y$ followed by an anticlockwise (resp. clockwise) arrow (resp. formal inverse) around $v'$ from $y$ to $z$. Finally, if a vertex $v$ has multiplicity $\mathfrak{e}_v$ and $v$ generates a cycle $\gamma_1\ldots\gamma_m$ in $Q$, then we generally cannot draw an anticlockwise cycle of arrows $(\gamma_1\ldots\gamma_m)^{\mathfrak{e}_v}$ or a clockwise cycle of formal inverses $(\gamma^{-1}_m\ldots\gamma^{-1}_1)^{\mathfrak{e}_v}$ around $v$ on the graph. An exception to this rule is the string given by a uniserial indecomposable projective. However, since we only wish to consider the stable Auslander-Reiten quiver of $A$, we will ignore this exception. In all other cases, we obtain a valid string. For examples of sequences of arrows and formal inverse presented on the graphs that are not valid strings, see Figure~\ref{StringNonEg}.
	 
	 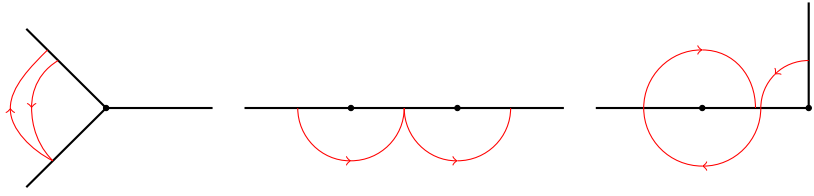
\begin{figure}[t]
	 	\centering
	 	\begin{tikzpicture}[scale=0.7]
			\draw[thick] (0.5,1.5) -- (2,0) -- (4,0);
			\draw[thick] (0.5,-1.5) -- (2,0) node (v3) {};
			\draw [red, ->](1.1,0.9) .. controls (0.6,0.6) and (0.6,0.1) .. (0.6,0);
			\draw [red](0.6,0) .. controls (0.6,-0.1) and (0.6,-0.6) .. (1,-1);
			\draw [red](0.9,1.1) .. controls (0.6,0.8) and (0.2,0.4) .. (0.2,0);
			\draw [red,<-](0.2,0) .. controls (0.2,-0.4) and (0.6,-0.8) .. (1,-1);

			\draw[thick] (4.6,0) -- (6.6,0) node (v1) {} -- (8.6,0) node (v2) {}-- (10.6,0);
	
			\draw [fill=black] (v1) ellipse (0.05 and 0.05);
			\draw [fill=black] (v2) ellipse (0.05 and 0.05);
			\draw [fill=black] (v3) ellipse (0.05 and 0.05);

			\draw [red, ->](5.6,0) arc (180:270:1);
			\draw [red](6.6,-1) arc (-90:0:1);
			\draw [red, ->](7.6,0) arc (180:270:1);
			\draw [red](8.6,-1) arc (-90:0:1);
			\draw [thick](11.2,0) -- (13.2,0) node (v4) {} -- (15.2,0) node (v5) {} -- (15.2,2);
			\draw [red](14.5636,0.6364) arc (135:180:0.9);
			\draw [red](13.2,-1.1) arc (-90:-180:1.1);
			\draw [red, ->](15.2,0.9) arc (90:135:0.9);
			\draw [red, ->](14.3,0) arc (0:-90:1.1);
			\draw [red, ->](12.1,0) arc (180:90:1.1);
			\draw [red](13.2,1.1) .. controls (13.8,1.1) and (14.2,0.6) .. (14.2,0);
			\draw [fill=black] (v4) ellipse (0.05 and 0.05);
			\draw [fill=black] (v5) ellipse (0.05 and 0.05);
		\end{tikzpicture}
	 	\caption{The above sequences of arrows and formal inverse presented on the graphs are not valid strings. All vertices are of multiplicity 1.} \label{StringNonEg}
	 \end{figure}
	 
	 \begin{rem} \label{HEStartEnd}
	 	For any string $\alpha\beta$ of length 2, we have $\widehat{e}(\alpha)=\widehat{s}(\beta)$ if and only if $\alpha\beta$ is direct or inverse. This follows from the fact that $\alpha\beta \in I$ if $\alpha\beta$ is direct/inverse and $\widehat{e}(\alpha)\neq\widehat{s}(\beta)$, since in this case, $\alpha$ does not directly follow $\beta$ in any cycle $\mathfrak{C}_v$ for any vertex $v$ in $G$. Conversely, $\alpha=\beta^{-1}$ if $\widehat{e}(\alpha)=\widehat{s}(\beta)$ but $\alpha\beta$ is not direct/inverse.
	 \end{rem}
	 
	 \subsection{Maximal direct and inverse strings}
	 To construct hooks and cohooks, we will need to know precisely when a string $w$ ends in a deep or on a peak. Specifically, we have the following.
	 
	 \begin{lem} \label{PeakDeep}
	 	Let $A$ be a Brauer graph algebra constructed from a graph $G$ and let $w$ be a string in $A$. Suppose $M(w)$ is non-projective. Then
	 	\begin{enumerate}[label=(\alph*)]
	 		\item $w$ ends in a deep if and only if either
	 		\begin{enumerate}[label=(\roman*)]
	 			\item $w=w_0\alpha^{-1}$, where $w_0$ is a string and $\alpha^{-1}\in Q_1^{-1}$ is a formal inverse such that $e(\alpha^{-1})$ is a truncated edge in $G$; or
	 			\item $w=w_0\gamma^{\mathfrak{e}_v-1}\gamma_1\ldots\gamma_{m-1}$, such that $w_0$ is a string, $\gamma=\gamma_1\ldots\gamma_m$ is a cycle generated by a vertex $v$ in $G$, and if $w_0$ is a zero string then $s(w)$ is not truncated.
	 		\end{enumerate}
	 		\item $w$ ends on a peak if and only if either
	 		\begin{enumerate}[label=(\roman*)]
	 			\item $w=w_0\alpha$, where $w_0$ is a string and $\alpha \in Q_1$ is an arrow such that $e(\alpha)$ is a truncated edge in $G$; or
	 			\item $w=w_0(\gamma^{-1})^{\mathfrak{e}_v-1}\gamma^{-1}_m\ldots\gamma^{-1}_2$, such that $w_0$ is a string, $\gamma=\gamma_1\ldots\gamma_m$ is a cycle generated by a vertex $v$ in $G$, and if $w_0$ is a zero string then $s(w)$ is not truncated.
	 		\end{enumerate}
	 	\end{enumerate}
	 \end{lem}
	 \begin{proof}
	 	(a) ($\Leftarrow:$) We will first prove (i) implies $w$ ends in a deep. So suppose $w=w_0\alpha^{-1}$ and $e(w)=x$, where $x$ is truncated. Then $P(x)$ is uniserial and thus there exists precisely one arrow of source $x$ and precisely one arrow of target $x$ in $Q$. But this implies the only arrow in $Q$ of source $x$ is $\alpha$. So there exists no arrow $\beta \in Q_1$ such that $w\beta$ is a string and hence, $w$ ends in a deep.
	 	
	 	Now suppose $w$ is instead of the form in (ii). We will show that this also implies $w$ ends in a deep. Let $x=s(\gamma_1)$. The first case to consider is where $x$ is non-truncated. In this case, there exists a relation $\gamma^{\mathfrak{e}_v} - \delta^{\mathfrak{e}_u}\in I$ for some cycle $\delta$ in $Q$ generated by a vertex $u$ connected to $x$. Thus, $\gamma^{\mathfrak{e}_v}$ is not a string. Furthermore, following the definition of a Brauer graph algebra, if $\gamma_i\beta\in I$ for some arrow $\beta$, then $\beta \neq \gamma_{i+1}$. Since there are no other relations involving each arrow $\gamma_i$ of the cycle $\gamma$, we conclude $\gamma^{\mathfrak{e}_v-1}\gamma_1\ldots\gamma_{m-1}$ is a string that ends in a deep.
	 	
	 	The other case to consider is where $x$ is truncated. Since $w$ is of the form in (ii), $w_0$ is not a zero string. Suppose for a contradiction that $w$ does not end in a deep. Then $w\gamma_m$ is a string, since $\gamma_{m-1}\beta \in I$ for any arrow $\beta \neq \gamma_m$. But then $w_0 \gamma^{\mathfrak{e}_v}$ is a string for some non-zero string $w_0$. This is not possible, since $\delta \gamma_1 \in I$ for any $\delta \neq \gamma_m$, and $\gamma_m\gamma^{\mathfrak{e}_v} \in I$. So $w$ must end in a deep.
	 	
	 	($\Rightarrow:$) Suppose conditions (i) and (ii) do not hold. If $w$ is a zero string, then $w$ clearly does not end in a deep. So let $w=\alpha_1\ldots\alpha_n$. First consider the case where $\alpha_n \in Q^{-1}_1$. Note that the arrow $\alpha_n^{-1}$ belongs to a cycle $\mathfrak{C}_v$ for some vertex $v$ in $G$. Since condition (i) does not hold, $P(e(w))$ is biserial and is the source of two distinct arrows in $Q$, so there exists an arrow $\beta$ not in $\mathfrak{C}_v$ such that $s(\beta)=e(\alpha_n)$. So clearly $w\beta$ is a string.
	 	
	 	Now suppose $\alpha_n \in Q_1$. Let $w''$ be the direct substring at the end of $w$ such that either $w=w''$ or $w=w'\beta^{-1}w''$ for some substring $w'$ and some $\beta^{-1}\in Q_1^{-1}$. Let $\gamma_1$ be the first symbol of $w''$. Then $\gamma_1$ belongs to a cycle of $Q$ generated by some vertex $v$ in $G$. Suppose $\gamma=\gamma_1\ldots\gamma_m$ is the cycle generated by $v$. Then $\alpha_n=\gamma_i$ for some $i$. Since condition (ii) does not hold, $w''$ forms a proper subword of $\gamma^{\mathfrak{e}_v-1}\gamma_1\ldots\gamma_{m-1}$ and $i \neq m-1$. Since $\gamma^{\mathfrak{e}_v-1}\gamma_1\ldots\gamma_{m-1}$ is a string, $w\gamma_{i+1}$ is a string and so $w$ does not end in a deep.
	 	
	 	(b) The proof is similar to (a).
	 \end{proof}
	 
	 We may also describe when a string $w$ presented on $G$ starts in a deep or on a peak -- we simply consider the string $w^{-1}$ in the context of the lemma above.
	 
	 \begin{rem} \label{MaxString}
	 	Suppose $x$ is a non-truncated edge of $G$ connected to a vertex $v$ and $\gamma=\gamma_1\ldots\gamma_n$ is the cycle of $Q$ generated by $v$ with $s(\gamma_1)=x$. Then it follows trivially from the above that a maximal direct string of source $x$ is $\gamma^{\mathfrak{e}_v-1}\gamma_1\ldots\gamma_{n-1}$ and a maximal inverse string of source $x$ is $(\gamma^{-1})^{\mathfrak{e}_v-1}\gamma^{-1}_n\ldots\gamma^{-1}_2$. Furthermore, to each non-truncated edge in $G$, we can associate two maximal direct strings and two maximal inverse strings -- one for each vertex connected to the edge (or in the case of a loop, one for each half-edge associated to the edge).
	 \end{rem}
	 
	 \subsection{Hooks and cohooks} \label{HooksCohooksSec}
	 Suppose a string $w$ does not end on a peak. If $w$ is a zero string, then there are at most two ways in which we can add a formal inverse to $w$. The precise number of ways we can do this for a zero string $w$ is determined by the number of arrows of target $e(w)$ in $Q$. Thus, if $e(w)$ is truncated then there is only one way in which we can add a formal inverse to $w$. Similarly, if $w$ is not a zero string, it follows from the definition of strings and the fact that $A$ is special biserial that there is only one way to add a formal inverse to the end of $w$.
	 
	 Adding a hook to the end of $w$ is by definition a string $w_h=w\alpha^{-1}\beta_1\ldots\beta_n$ that ends in a deep. For any given $w$, this string is necessarily unique unless $w$ is a zero string arising from a non-truncated edge in the Brauer graph. Let $e(\alpha^{-1})$ be the edge $\xymatrix@1{u \ar@{-}[r]^x & v}$ such that $\widehat{e}(\alpha^{-1})=x^u$. If $x$ is a truncated edge, then $w\alpha^{-1}$ ends in a deep by Lemma~\ref{PeakDeep}(a)(i) and hence $w_h=w\alpha^{-1}$. Otherwise, $\beta_1\ldots\beta_n$ is given by the maximal direct string around the vertex $v$ such that $\widehat{s}(\beta_1)=x^v$, described in Remark~\ref{MaxString}.
	 
	 Similarly, if $w$ does not end in a deep, then $w_c=w\alpha\beta^{-1}_1\ldots\beta^{-1}_n$ and $w_c$ ends on a peak. Let $e(\alpha)$ be the edge $\xymatrix@1{u \ar@{-}[r]^x & v}$ such that $\widehat{e}(\alpha)=x^u$. If $x$ is a truncated edge, then $w_c=w\alpha$ by Lemma~\ref{PeakDeep}(b)(i). Otherwise, $\beta^{-1}_1\ldots\beta^{-1}_n$ is given by the maximal inverse string around the vertex $v$ such that $\widehat{s}(\beta^{-1}_1)=x^v$, described in Remark~\ref{MaxString}.
	 
	 Informally, the process of adding or deleting hooks and cohooks to a string $w$ presented on a Brauer graph can be summarised as follows:
	 
	 \begin{enumerate}[label=(A\arabic*)]
	 	\item To add a hook to the end of $w$, we add a clockwise formal inverse to the end of $w$ around a vertex $u$ onto a connected edge $\xymatrix@1{u \ar@{-}[r]^x & v}$ and then, if $x$ is not truncated, add a maximal direct string of anticlockwise arrows around the vertex $v$. \label{AlAddHook}
	 	\item To add a cohook to the end of $w$, we add an anticlockwise arrow to the end of $w$ around a vertex $u$ onto a connected edge $\xymatrix@1{u \ar@{-}[r]^x & v}$ and then, if $x$ is not truncated, add a maximal inverse string of clockwise formal inverses around the vertex $v$. \label{AlAddCohook}
	 	\item To delete a hook from the end of $w$, we delete as many anticlockwise arrows as we can from the end of the string, and then we delete a single clockwise formal inverse. \label{AlDeleteHook}
	 	\item To delete a cohook from the end of $w$, we delete as many clockwise formal inverses as we can from the end of the string, and then we delete a single anticlockwise arrow. \label{AlDeleteCohook}
	 	\item To add or delete hooks or cohooks from the start of $w$, we follow (A1)-(A4) with the string $w^{-1}$. \label{AlStart}
	 \end{enumerate}
	 
	 Note however that there remains an ambiguity arising from zero strings associated to non-truncated edges in the Brauer graph. Suppose $w=\varepsilon_y$ with $y$ non-truncated and let $\alpha_1$ and $\alpha_2$ be distinct arrows of target $y$ in $Q$. Since $\alpha_1$ and $\alpha_2$ are distinct, $\widehat{e}(\alpha_1)\neq \widehat{e}(\alpha_2)$. When adding a hook to the end of $w$, we thus have a choice of $w_h=\alpha_1^{-1}w'$ or $w_h=\alpha_2^{-1}w''$ (for some maximal direct strings $w'$ and $w''$). If we choose $w_h=\alpha_1^{-1}w'$, then we must define $_hw=(w'')^{-1}\alpha_2$ (and vice versa). In terms of the Brauer graph, this reflects the choice of the vertex connected to $y$ around which we add a clockwise formal inverse. A similar choice exists for adding cohooks to $w$. If no other modules in the Auslander-Reiten component containing $S(y)$ are known, then this choice is simply a matter of notation. On the other hand, if other modules in the Auslander-Reiten component containing $S(y)$ are known, then we refer the reader to Subsection~\ref{LocSimpleProj} (and in particular Lemma~\ref{SimpleRay}).
	 
	 We further note from the almost split sequences given throughout \cite[Section 3]{butlerRingel}, that given a string module $M=M(w)$, the rays of source and target $M$ in the stable Auslander-Reiten quiver are given by Figure~\ref{ARRays}, where $\tau M(w_+)=M(w'_-)$, $\tau M(w'_+) = M(w_-)$ and 
	\begin{align*}
		w'_{-}&=
		\begin{cases}
			{_{-h}w} & \text{if } w \text{ starts in a deep,} \\
			{_cw} & \text{otherwise,}
		\end{cases}
		& w_{+}&=
		\begin{cases}
			w_{-c} & \text{if } w \text{ ends on a peak,} \\
			w_h & \text{otherwise,}
		\end{cases} &~ \\
		w_{-}&=
		\begin{cases}
			w_{-h} & \text{if } w \text{ ends in a deep,} \\
			w_c & \text{otherwise,}
		\end{cases}
		& w'_{+}&=
		\begin{cases}
			_{-c}w & \text{if } w \text{ starts on a peak,} \\
			_hw & \text{otherwise.}
		\end{cases} &~
	\end{align*}
	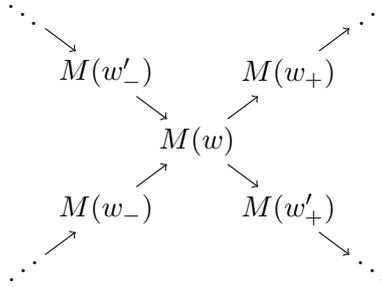
\begin{figure}[h]
		\centering
		\begin{tikzpicture}
			\draw (0,0) node {\small $M(w)$};
			\draw (-1.2,0.9) node {\small $M(w'_{-})$};
			\draw (1.2,0.9) node {\small $M(w_{+})$};
			\draw (-1.2,-0.9) node {\small $M(w_{-})$};
			\draw (1.2,-0.9) node {\small $M(w'_{+})$};
	
			\draw [->](-0.8,0.6) -- (-0.4,0.3);
			\draw [->](-0.8,-0.6) -- (-0.4,-0.3);
			\draw [->](0.4,-0.3) -- (0.8,-0.6);
			\draw [->](0.4,0.3) -- (0.8,0.6);
	
			\draw [->](-2,1.5) -- (-1.6,1.2);
			\draw [->](-2,-1.5) -- (-1.6,-1.2);
			\draw [->](1.6,-1.2) -- (2,-1.5);
			\draw [->](1.6,1.2) -- (2,1.5);
		
			\draw (-2.3,1.8) node {\small $\ddots$};
			\draw (2.3,-1.6) node {\small $\ddots$};
			\draw (2.3,1.8) node {\small $\iddots$};
			\draw (-2.3,-1.6) node {\small $\iddots$};
		\end{tikzpicture}
		\caption{The rays of source and target a given string module in the stable Auslander-Reiten quiver of a Brauer graph algebra.} \label{ARRays}
	\end{figure}
	
	Combining this information with (A1)-(A5) above, we obtain a constructive algorithm for reading off the stable Auslander-Reiten component containing $M$ from the Brauer graph.
	
	\begin{exam}
		With the Brauer tree algebra given in Figure~\ref{StringEg}, let $w$ be the string such that $M(w)=S(1)$. Then the string presented in Figure~\ref{StringEg} is the 4th module along the ray of target $S(1)$ in the direction of $M(w_-)$ of Figure~\ref{ARRays}.
	\end{exam}
	\begin{exam}
		A segment of a ray is presented in Figure~\ref{PathWalkExample} in Section~\ref{LocSimpleProj}. Here $w_0=w$ and $w_1=w_+$ in the context of Figure~\ref{ARRays}.
	\end{exam}
	
	Whilst the above algorithm is for constructing the stable Auslander-Reiten quiver, one can easily obtain the Auslander-Reiten quiver with the indecomposable projectives (which are also injective) using the almost split sequence
	\begin{equation*}
		\xymatrix@1{0 \ar[r] &\rad P \ar[r] & \rad P / \soc P \oplus P \ar[r] & P / \soc P \ar[r] & 0}
	\end{equation*}
	for each indecomposable projective $P$.
	
	\section{Components} \label{Components}
	Throughout this section, we will assume (unless specified otherwise) that our Brauer graph algebras are of infinite representation type. In \cite{ErdmannAR}, Erdmann and Skowro\'{n}ski classified the Auslander-Reiten components for self-injective special biserial algebras. For convenience, we shall restate part of their main results here.
	\begin{thm}[\cite{ErdmannAR}]
		Let $A$ be a special biserial self-injective algebra. Then the following are equivalent:
		\begin{enumerate}[label=(\roman*)]
			\item $A$ is representation-infinite domestic.
			\item $A$ is representation-infinite of polynomial growth.
			\item There are positive integers $m$, $p$, $q$ such that $_s\Gamma_A$ is a disjoint union of $m$ components of the form $\mathbb{Z}\tilde{A}_{p,q}$, $m$ components of the form $\mathbb{Z}A_\infty/\left<\tau^p\right>$, $m$ components of the form $\mathbb{Z}A_\infty/\left<\tau^q\right>$, and infinitely many components of the form $\mathbb{Z}A_\infty/\left<\tau\right>$.
		\end{enumerate}
	\end{thm}
	\begin{thm}[\cite{ErdmannAR}] \label{NonDomARComp}
		Let $A\cong KQ/I$ be a special biserial self-injective algebra. Then the following are equivalent:
		\begin{enumerate}[label=(\roman*)]
			\item $A$ is not of polynomial growth.
			\item $(Q, I)$ has infinitely many bands.
			\item $_s\Gamma_A$ is a disjoint union of a finite number of components of the form $\mathbb{Z}A_\infty/\left<\tau^n\right>$ with $n>1$, infinitely many components of the form $\mathbb{Z}A_\infty/\left<\tau\right>$, and infinitely many components of the form $\mathbb{Z}A_\infty^\infty$.
		\end{enumerate}
	\end{thm}
	
	Note that the $n$ in Theorem~\ref{NonDomARComp}(iii) can vary. That is, in a non-domestic self-injective special biserial algebra $A$, there exist integers $n_1,\ldots, n_r>1$ such that $_s\Gamma_A$ contains components of the form $\mathbb{Z}A_\infty/\left<\tau^{n_i}\right>$ for each $i$. For the purposes of this paper, we will distinguish between the exceptional tubes of rank 1 -- that is, the tubes of rank 1 consisting of string modules, of which there are finitely many -- and the homogeneous tubes of rank 1, which consist solely of band modules.
	
	It follows from the above that $A$ is domestic if and only if there are finitely many distinct bands in $A$. If $A$ contains $n$ distinct bands, then we refer to $A$ as an $n$-domestic algebra. In \cite{bocianSkow}, Bocian and Skowro\'{n}ski described the symmetric special biserial algebras that are domestic. Specifically, those that are 1-domestic are associated to a Brauer graph that is either a tree with precisely two exceptional vertices of multiplicity 2 and with all other vertices of multiplicity 1, or a graph with a unique simple cycle of odd length and with all vertices of multiplicity 1. Those that are 2-domestic are associated to a Brauer graph with a unique simple cycle of even length and with all vertices of multiplicity 1. All other graphs (with the exception of Brauer trees) produce non-domestic algebras.
	
	\subsection{Exceptional tubes} \label{ExcepTubesSec}
	We are interested in counting the number of exceptional tubes in $_s\Gamma_A$. To do this we must use a definition from \cite{Roggenkamp} for a set of modules, which will soon be of particular importance.
	
	\begin{defn}[\cite{Roggenkamp}]
		Let $\mathcal{M}$ be the set of all simple modules whose projective covers are uniserial and all maximal uniserial submodules of indecomposable projective $A$-modules.
	\end{defn}
	
	We note from \cite{Roggenkamp} that $\Omega(M) \in \mathcal{M}$ for all $M \in \mathcal{M}$ and the minimal projective resolutions of $M \in \mathcal{M}$ are periodic, which will be required in the proof of the following.
	
	\begin{lem} \label{tubeMouth}
		Let $A$ be a Brauer graph algebra. An indecomposable string module $M$ is at the mouth of a tube in the stable Auslander-Reiten quiver $_s\Gamma_A$ of $A$ if and only if $M \in \mathcal{M}$.
	\end{lem}
	\begin{proof}
		$(\Rightarrow:)$ Let $M$ be at the mouth of a tube in $_s\Gamma_A$. Then there exists precisely one irreducible morphism $M \rightarrow N$ in $_s\Gamma_A$ for some indecomposable $A$-module $N$. Suppose for a contradiction that $M$ is not uniserial and let $w=\alpha_1\ldots\alpha_m$ be the string such that $M(w)=M$. Then $w$ is not direct (or inverse). If $w$ ends on a peak, then deleting a cohook from the end of $w$ gives a string $w_{-c}$. Note that in the case where $w$ consists only of an arrow followed by an inverse string, then $w_{-c}$ is a zero string and $M(w_{-c})$ is simple. Similarly, if $w$ starts on a peak then deleting a cohook from the start of $w$ gives a string $_{-c}w$ such that $M(_{-c}w)$ is indecomposable. By the results of \cite{butlerRingel}, the Auslander-Reiten sequence starting in $M$ therefore has two non-projective middle terms. This implies that there are two irreducible morphisms of source $M$ in $_s\Gamma_A$ -- a contradiction to our assumption that $M(w)$ is not uniserial. It is easy to see that the same contradiction occurs in the cases where $w$ does not start or end on a peak. Therefore $w$ must be direct (or inverse) and $M$ is uniserial.
		
		Suppose $M$ is a simple module corresponding to the stationary path $\varepsilon_x$ in $Q$. Then the injective envelope $I(M)$ must be uniserial, since if $I(M)$ is biserial, then there are two distinct arrows of target $x$ in $Q$ and hence two non-projective middle terms in the Auslander-Reiten sequence starting in $M$. Thus, if $M$ is a simple module at the mouth of a tube in $_s\Gamma_A$ then $M \in \mathcal{M}$.
		
		Now suppose $M$ is not simple. For a contradiction, suppose that $M \not\in \mathcal{M}$. Note that this implies that $M$ is not the radical of a uniserial indecomposable projective. Let $w$ be the inverse string such that $M(w)=M$. Then $M$ is not maximal and therefore the inverse string $w$ is not maximal, which implies $w$ does not end on a peak. Hence, the Auslander-Reiten sequence starting in $M$ has an indecomposable middle term $M(w_h)$. If $w$ does not start on a peak, then the there exists another middle term $M({_hw})$ in the Auslander-Reiten sequence. Otherwise, if $w$ starts on a peak, then $_{-c}w=\alpha_2\ldots\alpha_m$. Thus, the Auslander-Reiten sequence starting in $M$ must have two non-projective middle terms by the results of \cite{butlerRingel} -- a contradiction. So $M \in \mathcal{M}$.
		
		$(\Leftarrow:)$ Suppose $M \in \mathcal{M}$. We will show that the Auslander-Reiten sequence starting in $M$ has precisely one non-projective middle term. If $M$ is the radical of a uniserial indecomposable projective-injective $P$ then the Auslander-Reiten sequence starting in $M$ is of the form
		\begin{equation*}
			\xymatrix@1{0 \ar[r] & M \ar[r] & M / \soc(P) \oplus P \ar[r] & P / \soc(P) \ar[r] & 0}
		\end{equation*}
		and $M / \soc(P)$ is indecomposable since $P$ is uniserial. Thus, there exists precisely one irreducible morphism of source $M / \soc(P)$ in $_s\Gamma_A$, as required. So suppose instead $M \in \mathcal{M}$ is not the radical of a uniserial indecomposable projective and let $w_0$ be the direct string such that $M(w_0)=M$. By \cite{Roggenkamp}, $\Omega^{-1}(M) \in \mathcal{M}$. So $\Omega^{-1}(M)$ is associated to a maximal direct string $w_1$ and $s(w_1)=e(w_0)$, which follows from the fact that $A$ is symmetric and $\Omega^{-1}(M)$ is a maximal uniserial quotient of $I(M)$. Also note that $\Omega^{-1}(M)$ must be non-simple, since this would otherwise contradict our assumption that $M$ is not the radical of a uniserial indecomposable projective. Thus, there exists an arrow $\beta \in Q_1$ such that $s(\beta)=e(w_1)$ and $e(\beta)=e(w_0)=s(w_1)$. Now let $w_2$ be the (maximal) direct string associated to $\Omega^{-2}(M) \in \mathcal{M}$. Then by a similar argument used for the strings $w_0$ and $w_1$, we have $s(w_2)=e(w_1)$. Thus, there exists a string $w_0\beta^{-1}w_2$ and the sequence
		\begin{equation*}
			\xymatrix@1{0 \ar[r] & M(w_0) \ar[r]^-f & M(w_0\beta^{-1}w_2) \ar[r]^-g & M(w_2) \ar[r] & 0}
		\end{equation*}
		is exact. Note that in the case that $M$ (resp. $\Omega^{-2}(M)$) is simple, the string $w_0$ (resp. $w_2$) is a zero string. Now $w_0\beta^{-1}w_2$ is obtained from $w_0$ by adding a hook and $w_0\beta^{-1}w_2$ is obtained from $w_2$ by adding a cohook. So $f$ and $g$ are irreducible and hence, the above sequence is an Auslander-Reiten sequence.
	\end{proof}
	
	To each edge $x$ in a Brauer graph $G$, we can associate precisely two modules $M_1,M_2 \in \mathcal{M}$. These two modules have the property that $\tp(M_1)=\tp(M_2)=x$. For each half-edge $x^v$ associated to $x$ and incident to a non-truncated vertex $v$, we have a string module $M(w) \in \mathcal{M}$ such that $w=(\gamma_1\ldots\gamma_n)^{\mathfrak{e}_v-1}\gamma_1\ldots\gamma_{n-1}$, where $\mathfrak{C}_{v,\gamma_1}=\gamma_1\ldots\gamma_n$ and $\widehat{s}(\gamma_1)=x^v$. If $x$ is incident to a truncated vertex $u$, then then $P(x)$ is uniserial and we associate to $x^u$ the module $S(x) \in \mathcal{M}$. Thus, each half-edge in $G$ corresponds (bijectively) to a module in $\mathcal{M}$.
	
	We further recall from \cite{Roggenkamp} that $\Omega(M) \in \mathcal{M}$ for all $M \in \mathcal{M}$ and that the minimal projective resolution of a module in $\mathcal{M}$ follows a Green walk (described in \cite{GreenWalk}) of the Brauer graph. Specifically, it follows from \cite[Remark 3.6]{Roggenkamp} that if $M \in \mathcal{M}$ corresponds to a half-edge $x^v$ in $G$ (in the sense described above) and the $i$-th step along a Green walk from $\overline{x^v}$ is a half-edge $\overline{x_i^{v_i}}$, then $\Omega^i(M) \in \mathcal{M}$ corresponds to the half-edge $x_i^{v_i}$.
	 
	 The relation $M \sim N$ for $M,N \in \mathcal{M} \Leftrightarrow N=\Omega^{2i}(M)$ for some $i$ is an equivalence relation. Thus, we can partition the set $\mathcal{M}$ of a Brauer graph algebra into orbits described by the distinct double-stepped Green walks of the Brauer graph (a similar relation can be constructed for single-stepped Green walks, but this is not of interest to us). Since Brauer graph algebras are symmetric, which implies $\tau=\Omega^2$, a consequence of Lemma~\ref{tubeMouth} is the following.
	 
	 \begin{thm} \label{tubeWalks}
	 	Let $A$ be a representation-infinite Brauer graph algebra with Brauer graph $G$ and let $_s\Gamma_A$ be its stable Auslander-Reiten quiver. Then
	 	\begin{enumerate}[label=(\alph*)]
	 		\item there is a bijective correspondence between the exceptional tubes in $_s\Gamma_A$ and the distinct double-stepped Green walks of $G$; and
	 		\item the rank of an exceptional tube is given by the length of the corresponding double-stepped Green walk of $G$.
	 	\end{enumerate}
	 \end{thm}

	\subsection{Domestic Brauer graph algebras}
	For Brauer graph algebras that are infinite-domestic, the above theorem allows us to determine the precise shape of the $\mathbb{Z}\tilde{A}_{p,q}$ components. 
	
	\begin{thm}
		Let $A$ be a Brauer graph algebra constructed from a graph $G$ of $n$ edges and suppose $_s\Gamma_A$ has a $\mathbb{Z}\tilde{A}_{p,q}$ component.
		\begin{enumerate}[label=(\alph*)]
			\item If $A$ is 1-domestic, then $p+q=2n$.
			\item If $A$ is 2-domestic, then $p+q=n$.
		\end{enumerate}
		Furthermore, if $G$ is a tree, then $p=q=n$.
	\end{thm}
	\begin{proof}
		Suppose $A$ is a 1-domestic Brauer graph algebra constructed from a tree. For any tree, there is only one possible single-stepped Green walk. This walk steps along each edge exactly twice -- that is, we walk along both sides of each edge -- and thus, the walk is of length $2n$, which is even. This therefore amounts to two distinct double-stepped walks of length $n$. By Theorem~\ref{tubeWalks} and \cite[Theorem~2.1]{ErdmannAR}, $p=q=n$.
		
		Suppose instead we have a Brauer graph consisting solely of a simple cycle of odd length, say $l$. Then there are two distinct single-stepped Green walks of odd length $l$ -- one along the `inside' of the cycle and one along the `outside'. Since Green walks are periodic, this amounts to two distinct double-stepped Green walks of length $l$. Inserting an additional edge to any vertex in the graph will add another two steps to one of the two distinct single-stepped Green walks, and thus, an additional two steps along one of the two distinct double-stepped Green walks. Proceeding inductively, the result for (a) follows.
		
		The argument for (b) is similar, except our initial simple cycle of length $l$ is even. The two distinct single-stepped Green walks are then of even length, and hence each split into a pair of distinct double-stepped Green walks of length $\frac{l}{2}$, making a total of four distinct double-stepped Green walks.  Inserting an additional edge to any vertex in the graph will add a single additional step along a pair of the four distinct double-stepped Green walks. Proceeding inductively, the result follows.
	\end{proof}
	
	The above proof implies that by distinguishing between the two sides of a cycle (an `inside' and an `outside') in a domestic Brauer graph algebra, one can actually read off $p$ and $q$ directly. To be more precise, if a Brauer graph $G$ contains a unique simple cycle, then $G$ will have two distinct (single-stepped) Green walks. When $G$ is presented as a planar graph, the half-edges along one Green walk will appear to be inside the cycle, and the half-edges along the other Green walk will appear to be outside the cycle. For example, see Figure~\ref{walkEx}. For an edge that is not in the cycle of $G$, we can say that it is an additional edge on the inside (resp. outside) of the cycle if its associated half-edges occur along the Green walk inside (resp. outside) the cycle.
	
	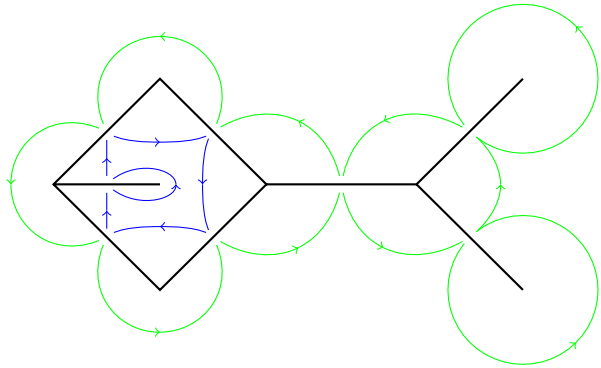
\begin{figure}[h]
		\centering
		\begin{tikzpicture}[scale=0.7]
			\draw [green] (1,1) .. controls (0.5,2) and (1.2,2.8)  .. (2,2.8);
			\draw [green][->] (3,1) .. controls (3.5,2) and (2.8,2.8)  .. (2,2.8);
			\draw [green](3,1) .. controls (3.6,1.4) and (4.2,1.4)  .. (4.6,1.2);
			\draw [green][->] (5.41,0) .. controls (5.3,0.6) and (5,1)  .. (4.6,1.2);
			\draw [green](5.41,0) .. controls (5.5,0.6) and (5.8,1)  .. (6.2,1.2);
			\draw [green][->] (7.82,1) .. controls (7.2,1.4) and (6.6,1.4)  .. (6.2,1.2);
			\draw [green][->] (7.82,-1) .. controls (8.1,-0.8) and (8.4,-0.4)  .. (8.4,0);
			\draw [green](7.82,1) node (v8) {} .. controls (8.1,0.8) and (8.4,0.4)  .. (8.4,0);
			\draw [green][->] (5.41,0) .. controls (5.5,-0.6) and (5.8,-1)  .. (6.2,-1.2);
			\draw [green](7.82,-1) node (v7) {} .. controls (7.2,-1.4) and (6.6,-1.4)  .. (6.2,-1.2);
			\draw [green][->] (1,-1) .. controls (0.5,-2) and (1.2,-2.8)  .. (2,-2.8);
			\draw [green](3,-1) .. controls (3.5,-2) and (2.8,-2.8)  .. (2,-2.8);
			\draw [green][->] (3,-1) .. controls (3.6,-1.4) and (4.2,-1.4)  .. (4.6,-1.2);
			\draw [green](5.41,0) node (v6) {} .. controls (5.3,-0.6) and (5,-1)  .. (4.6,-1.2);
			\draw [green] (1,-1) .. controls (0,-1.5) and (-0.8,-0.8)  .. (-0.8,0);
			\draw [green][<-](-0.8,0) .. controls (-0.8,0.8) and (0,1.5)  .. (1,1);
			\draw [green][<-] [domain=45:-135] plot ({8.82+1.41*cos(\x)}, {2+1.41*sin(\x)});
			\draw [green][domain=225:45] plot ({8.82+1.41*cos(\x)}, {2+1.41*sin(\x)});
			\draw [green][<-] [domain=-45:-225] plot ({8.82+1.41*cos(\x)}, {-2+1.41*sin(\x)});
			\draw [green][domain=135:-45] plot ({8.82+1.41*cos(\x)}, {-2+1.41*sin(\x)});

			\draw [blue](3,-1) .. controls (2.8,-0.8) and (2.8,-0.2) .. (2.8,0);
			\draw [blue][->] (3,1) .. controls (2.8,0.8) and (2.8,0.2) .. (2.8,0);
			\draw [blue](1,-1) .. controls (1.2,-0.8) and (1.8,-0.8) .. (2,-0.8);
			\draw [blue][->](3,-1) node (v3) {} .. controls (2.8,-0.8) and (2.2,-0.8) .. (2,-0.8);
			\draw [blue][->](1,1) .. controls (1.2,0.8) and (1.8,0.8) .. (2,0.8);
			\draw [blue](3,1) node (v4) {} .. controls (2.8,0.8) and (2.2,0.8) .. (2,0.8);
			\draw [blue][->] (1,-1) node (v2) {} -- (1,-0.5);
			\draw [blue][->] (1,-0.5) -- (1,0.5);
			\draw [blue](1,0.5) -- (1,1) node (v1) {};
			\draw [blue][->] (1,0) .. controls (1.5,-0.5) and (2.3,-0.3) .. (2.3,0);
			\draw [blue](1,0) node (v5) {} .. controls (1.5,0.5) and (2.3,0.3) .. (2.3,0);
			
			\draw [white,fill=white] (v1) ellipse (0.15 and 0.15);
			\draw [white,fill=white] (v2) ellipse (0.15 and 0.15);
			\draw [white,fill=white] (v3) ellipse (0.15 and 0.15);
			\draw [white,fill=white] (v4) ellipse (0.15 and 0.15);
			\draw [white,fill=white] (v5) ellipse (0.15 and 0.15);
			\draw [white,fill=white] (v6) ellipse (0.15 and 0.15);
			\draw [white,fill=white] (v7) ellipse (0.15 and 0.15);
			\draw [white,fill=white] (v8) ellipse (0.15 and 0.15);
			
			\draw[thick] (0,0) -- (2,2) -- (4,0) -- (2,-2) -- cycle
			 	(0,0) -- (2, 0)	
				(4,0) -- (6.82,0) -- (8.82,2)
				(6.82,0) -- (8.82,-2);
		\end{tikzpicture}
		\caption{A 2-domestic Brauer graph algebra with distinct Green walks on the inside (blue) and outside (green) of the cycle in the Brauer graph.} \label{walkEx}
	\end{figure}
	
	\begin{cor} \label{TubeType}
		 For a domestic Brauer graph algebra, if the Brauer graph contains a unique simple cycle of length $l$ and there are $n_1$ additional edges on the inside of the cycle and $n_2$ additional edges along the outside, then the $\mathbb{Z}\tilde{A}_{p,q}$ components are given by
		\begin{equation*}
			p=
			\begin{cases}
				l+2n_1 & l \text{ odd,} \\
				\frac{l}{2}+n_1 & l \text{ even,}
			\end{cases}
			\quad \text{and} \quad
			q = 
			\begin{cases}
				l+2n_2 & l \text{ odd,} \\
				\frac{l}{2}+n_2 & l \text{ even.}
			\end{cases}
		\end{equation*}
	\end{cor}
	
	\begin{exam}
		In Figure~\ref{walkEx}, we have $l=4$, $n_1 = 1$ and $n_2 = 3$. So the Auslander-Reiten quiver has two components $\mathbb{Z}\tilde{A}_{3,5}$, two components $\mathbb{Z}A_\infty/\left<\tau^3\right>$, two components $\mathbb{Z}A_\infty/\left<\tau^5\right>$ and two infinite families of homogeneous tubes.
	\end{exam}
	
	\subsection{The locations of simples and projectives} \label{LocSimpleProj}
		It is possible to determine exactly which simple modules belong to the exceptional tubes of $_s\Gamma_A$ by looking at the Brauer graph. Similarly, one can also determine the location in the stable Auslander-Reiten quiver of the radical of a projective $P$ by the edge associated to $P$ in the Brauer graph.
	
	In \cite{butlerRingel}, the irreducible morphisms in $_s\Gamma_A$ between string modules is described. In order to consider whether a string module $M(w)$ belongs to an exceptional tube, we will need to look at the rays in the Auslander-Reiten quiver of source and target $M(w)$, of which there are at most two each. If at least one of these rays terminates, then $M(w)$ must belong to an exceptional tube. Otherwise, if all rays are infinite, then $M(w)$ must belong to either a $\mathbb{Z}\tilde{A}_{p,q}$ component (if $A$ is domestic) or a $\mathbb{Z}A_\infty^\infty$ component (if $A$ is non-domestic). Therefore in what follows, we will need to consider the bi-directional sequence $(w_j)_{j \in \mathbb{Z}}$ defined by
	\begin{equation} \tag{$\ast$} \label{StrSeq}
		\begin{split}
			w_{j+1}&=
			\begin{cases}
				(w_j)_{-c} & \text{if } w_j \text{ ends on a peak,} \\
				(w_j)_h & \text{otherwise,} 
			\end{cases} \\
			w_{j-1}&=
			\begin{cases}
				(w_j)_{-h} & \text{if } w_j \text{ ends in a deep,} \\
				(w_j)_c & \text{otherwise,}
			\end{cases}
		\end{split}
	\end{equation}
	with $w_0=w$. We will refer to this sequence frequently throughout this subsection of the paper. Any string module $M(w_i)$ then lies along the line through $M(w)$ in $_s\Gamma_A$ given by adding or deleting from the end of $w$. This is precisely the south west to north east line through $M(w)$ illustrated in Figure~\ref{ARRays} at the end of Section~\ref{Algorithm}. To obtain the line through $M(w)$ in $_s\Gamma_A$ given by adding or deleting from the start of $w$ (the north west to south east line through $M(w)$ illustrated in Figure~\ref{ARRays}), we need only consider the sequence (\ref{StrSeq}) with $w_0 = w^{-1}$.
	
	Note that in the case where $w_0$ is the zero string $\varepsilon_x$ at a non-truncated edge $\xymatrix@1{u \ar@{-}[r]^x & v}$, there exists an ambiguity in the sequence (\ref{StrSeq}), as there are two possible ways in which we can add a hook (resp. cohook) to $w_0$. We may either add a hook (resp. cohook) starting with the formal inverse (resp. arrow) $\alpha$ such that $\widehat{s}(\alpha)=x^u$, or with the formal inverse (resp. arrow) $\alpha'$ such that $\widehat{s}(\alpha')=x^v$. In this case, we must define either $w_1$ or $w_{-1}$. The following shows that it is sufficient to define only one of these terms.
	
	\begin{lem} \label{SimpleRay}
		Let $A=KQ/I$ be a Brauer graph algebra associated to a Brauer graph $G$, and let $x^v$ be a half-edge associated to an edge $x$ incident to a vertex $v$ in $G$. Suppose there exists a segment
		\begin{equation*}
			\cdots \rightarrow M(w) \rightarrow S(x) \rightarrow M(w')  \rightarrow \cdots
		\end{equation*}
		of a line $L$ in $_s\Gamma_A$ for some strings $w=(\varepsilon_x)_{c}$ and $w'=(\varepsilon_x)_{h}$. If the first symbol $\alpha$ of $w$ is such that $\widehat{s}(\alpha)=\overline{x^v}$, then the first symbol $\beta^{-1}$ of $w'$ is such that $\widehat{s}(\beta^{-1})=x^v$.
	\end{lem}
	\begin{proof}
		There exist two irreducible morphisms of source $S(x)$ in $_s\Gamma_A$. Suppose $S(x) \rightarrow M(w'')$ is an irreducible morphism such that $M(w'') \not\cong M(w')$. Then $w''={_h(\varepsilon_x)}$ and the morphism $S(x) \rightarrow M(w'')$ does not belong to the line $L$. In particular, $M(w'')=\tau^{-1} M(w)$. Let $\gamma$ be the last symbol of $w''$ and let $\beta^{-1}$ be the first symbol of $w'$. It is sufficient to show that $\widehat{e}(\gamma)=\overline{x^v}$, since it then follows that $\widehat{s}(\beta^{-1}) = x^v$.
		
		To calculate $w''$, we will investigate the Auslander-Reiten sequences starting in $M(w)$. Let $y=e(\alpha)$. The string $w$ necessarily ends on a peak. Suppose $w$ also starts on a peak. Then $M(w)$ is a maximal submodule of $P(y)$ (since $w$ is obtained by adding a cohook to a zero string), so $M(w)=\rad P(y)$. Note that $P(y)$ must be biserial, since if $P(y)$ were uniserial then $P(y)$ would be isomorphic to $M(w_0\alpha)$ for some non-zero string $w_0$ (and hence, our assumption that $w$ starts on a peak would be false). Thus, there exists an Auslander-Reiten sequence
		\begin{equation*}
			0 \rightarrow \rad P(y) \rightarrow \rad P(y) / \soc P(y) \oplus P(y) \rightarrow P(y) / \soc P(y) \rightarrow 0.
		\end{equation*}
		with three middle terms (since $\rad P(y) / \soc P(y)$ is a direct sum of two uniserial modules) and $M(w'') = P(y) / \soc P(y)$. Let $w=\alpha\delta^{-1}_n\ldots\delta^{-1}_1$. Then necessarily, $w''=\delta^{-1}_{n-1}\ldots\delta^{-1}_1\delta^{-1}_0 \gamma$ for some formal inverse $\delta^{-1}_0$ and some arrow $\gamma$ such that $e(\delta^{-1}_0)=y=s(\gamma)$. Note that by Remark~\ref{HEStartEnd}, $\widehat{e}(\alpha)\neq\widehat{s}(\delta^{-1}_n)$, $\widehat{e}(\delta^{-1}_i)=\widehat{s}(\delta^{-1}_{i+1})$ and $\widehat{e}(\delta^{-1}_0)\neq\widehat{s}(\gamma)$. Since $\widehat{s}(\delta^{-1}_n)$, $\widehat{e}(\delta^{-1}_0)$, $\widehat{e}(\alpha)$ and $\widehat{s}(\gamma)$ are all half-edges associated to $y$ (of which there are precisely two) and $e(\gamma)=x=s(\alpha)$, we conclude that $\widehat{e}(\gamma)=\widehat{s}(\alpha)=\overline{x^v}$, as required.
		
		Suppose instead that $w$ does not start on a peak. Then by \cite{butlerRingel}, there exists an Auslander-Reiten sequence
		\begin{equation*}
			0 \rightarrow M(w) \rightarrow M(w_{-c}) \oplus M({_hw}) \rightarrow M({_{h}w_{-c}}) \rightarrow 0.
		\end{equation*}
		So $M(w'') \cong M({_{h}w_{-c}})$. Again let $\gamma$ be the last symbol of $w''$ and $\alpha$ be the first symbol of $w$. Then $\gamma\alpha$ is a direct substring of $_hw$. Thus, $\widehat{e}(\gamma)=\widehat{s}(\alpha)=\overline{x^v}$ by Remark~\ref{HEStartEnd}, as required. The result then follows.
	\end{proof}
	
	To prove the results of this section, it will be helpful to track the end of the strings along a ray of source or target a given module in $_s\Gamma_A$.
	
	\begin{prop} \label{MorphismGreenWalk}
		Let $A=KQ/I$ be a Brauer graph algebra associated to a Brauer graph $G$. Let $w_0$ be a string such that either $w_0$ is the zero string $\varepsilon_x$ or $w_0=\alpha_1\ldots\alpha_n$, where $\widehat{e}(\alpha_n)=x^v$ and $x^v$ is a half-edge associated to an edge $x$ connected to a vertex $v$ in $G$.
		\begin{enumerate}[label=(\alph*)]
			\item Suppose there exists a ray in $_s\Gamma_A$
			\begin{equation*}
				M(w_0) \rightarrow M(w_1) \rightarrow \ldots \rightarrow M(w_k) \rightarrow \cdots
			\end{equation*}
			such that $w_1$ is given by adding or deleting from the end of $w_0$.
			\begin{enumerate}[label=(\roman*)]
				\item If $\alpha_n \in Q^{-1}_1$ (resp. $\alpha_n \in Q_1$) then the edge $e(w_i)$ corresponds to the half-edge at the $i$-th step along a clockwise double-stepped Green walk from $x^v$ (resp. $\overline{x^v}$) for all $i \leq k$.
				\item If $w_0=\varepsilon_x$ and $w_1=\beta^{-1}\gamma_1\ldots\gamma_m$, where $\widehat{s}(\beta^{-1})=x^v$, then the edge $e(w_i)$ corresponds to the half-edge at the $i$-th step along a clockwise double-stepped Green walk from $x^v$ for all $i \leq k$.
			\end{enumerate}
	
			\item Suppose there exists a ray in $_s\Gamma_A$
			\begin{equation*}
				\cdots \rightarrow M(w_{-k}) \rightarrow \ldots \rightarrow M(w_{-1}) \rightarrow M(w_0)
			\end{equation*}
			such that $w_1$ is given by adding or deleting from the end of $w_0$.
			\begin{enumerate}[label=(\roman*)]
				\item If $\alpha_n \in Q_1$ (resp. $\alpha_n \in Q^{-1}_1$) then the edge $e(w_i)$ corresponds to the half-edge at the $i$-th step along a double-stepped Green walk from $x^v$ (resp. $\overline{x^v}$) for all $i \leq k$.
				\item If $w_0=\varepsilon_x$ and $w_{-1}=\beta\gamma^{-1}_1\ldots\gamma^{-1}_m$, where $\widehat{s}(\beta)=x^v$, then the edge $e(w_{-i})$ corresponds to the half-edge at the $i$-th step along a double-stepped Green walk from $x^v$ for all $i \leq k$.
			\end{enumerate}
		\end{enumerate}
	\end{prop}
	\begin{proof}
		(a) Let $x_0^{v_0}$ be a half-edge associated to the edge $x=x_0$ and label the $i$-th step along a double-stepped clockwise Green walk from $x_0^{v_0}$ as $x_i^{v_i}$. Let $\alpha$ be the last symbol of $w_i$. Assume that one of the following holds for the string $w_i$.
		\begin{enumerate}[label=(\roman*)]
			\item $\alpha \in Q^{-1}_1$ and $\widehat{e}(\alpha) = x_i^{v_i}$
			\item $\alpha \in Q_1$ and $\widehat{e}(\alpha) = \overline{x_i^{v_i}}$
			\item $w_i$ is the zero string $\varepsilon_{x_i}$
		\end{enumerate}
		We aim to show that the string $w_{i+1}$ satisfies the analogous properties of (i)-(iii). The proof of this claim requires us to investigate multiple cases related to the end of $w_i$.
		
		Case 1: $\alpha \in Q^{-1}_1$ and $w_i$ ends on a peak. By Lemma~\ref{PeakDeep}(b)(ii), there must exist a maximal inverse substring $w'=\gamma^{-1}_1\ldots\gamma^{-1}_r$ at the end of $w_i$. Moreover, it follows from the maximality of $w'$ that $\widehat{s}(\gamma^{-1}_1)=y^{v_i}$, where $y^{v_i}$ is the predecessor to $x_i^{v_i}$. Since $w_i$ ends on a peak, $w_{i+1}$ is given by deleting a cohook from the end of $w_i$. Thus, $w_i=w_{i+1}\beta w'$ for some arrow $\beta$ (which exists since otherwise $M(w_i) \in \mathcal{M}$ and $w_{i+1}$ would not be defined). Necessarily, $\widehat{s}(\beta)=z^u$, where $z^u$ is the predecessor to $\overline{y^{v_i}}$. This is illustrated in Figure~\ref{MorphWalk13}(a).
		
		The first two steps along a clockwise Green walk from $x_i^{v_i}$ are $\overline{y^{v_i}}$ and $\overline{z^u}$. Thus, $\overline{z^u}=x_{i+1}^{v_{i+1}}$. If $w_{i+1}$ is a zero string, then $e(w_{i+1})=x_{i+1}$ and hence, $w_{i+1}=\varepsilon_{x_{i+1}}$. So $w_{i+1}$ satisfies property (iii) at the start of the proof. Otherwise, let $\gamma$ be the last symbol of $w_{i+1}$. If $\gamma \in Q^{-1}_1$, then we necessarily have $\widehat{e}(\gamma)=\overline{z^u}=x_{i+1}^{v_{i+1}}$. So $w_{i+1}$ satisfies property (i). If $\gamma \in Q_1$, then $\widehat{e}(\gamma)=z^u=\overline{x_{i+1}^{v_{i+1}}}$ and hence, $w_{i+1}$ satisfies property (ii).
		
		\begin{figure}[b]
			\centering
			\begin{tikzpicture}[scale=1.1]
				\draw (-6.2,-0.4) node {(a) Case 1:};
				\draw (-4.7,-1) -- (-3.9,-0.4) -- (-2.9,-0.4) -- (-2.1,0.2);
				\draw (-2.1,-1) -- (-2.9,-0.4);
				\draw (-2.5,-0.3) node {$\vdots$};
				\draw (-2.9,-0.2) node {$v_i$};
				\draw (-3.9,-0.2) node {$u$};
				\draw (-2.66,-0.87) node {$x_i$};
				\draw (-3.2,-0.6) node {$y$};
				\draw (-4.32,-0.96) node {$z$};
				\draw [red] (-4.8,-0.7) node {$w_{i+1}$};
				\draw [blue] (-3.6,-1.1) node {$\beta$};
				\draw [cyan] (-2.9,0.3) node {$w'$};
				\draw [fill=black] (-2.1,0.2) ellipse (0.04 and 0.04);
				\draw [fill=black] (-2.9,-0.4) ellipse (0.04 and 0.04);
				\draw [fill=black] (-2.1,-1) ellipse (0.04 and 0.04);
				\draw [fill=black] (-3.9,-0.4) ellipse (0.04 and 0.04);
				\draw [fill=black] (-4.7,-1) ellipse (0.04 and 0.04);
				\draw [->, blue](-4.3045,-0.6939) arc (-143.9987:-57.6:0.5);
				\draw [blue](-3.6321,-0.8222) arc (-57.6035:0:0.5);
				\draw [->,cyan](-3.4,-0.4) arc (-180:-252:0.5);
				\draw [->,cyan](-3.0545,0.0755) arc (108.0001:0:0.5);
				\draw [cyan](-2.4,-0.4) arc (0:-36:0.5);
				\draw [->,red](-4.5,-0.5) -- (-4.4,-0.6);
				\draw [red](-4.4,-0.6) -- (-4.3,-0.7);
				\draw [red, dotted](-4.7,-0.3) -- (-4.5,-0.5);

				\draw (-0.6,-0.5) node {(b) Case 3:};
				\draw (0.9,-1) -- (1.7,-0.4) -- (2.7,-0.4);
				\draw (1.7,-0.2) node {$u$};
				\draw (2.7,-0.2) node {$v_i$};
				\draw (2.2,-0.2) node {$x_i$};
				\draw (1.2727,-0.9727) node {$y$};
				\draw [red] (0.8,-0.7) node {$w_{i+1}$};
				\draw [blue] (2,-1.1) node {$\alpha$};
				\draw [fill=black] (2.7,-0.4) ellipse (0.04 and 0.04);
				\draw [fill=black] (1.7,-0.4) ellipse (0.04 and 0.04);
				\draw [fill=black] (0.9,-1) ellipse (0.04 and 0.04);
				\draw [->, blue](1.2955,-0.6939) arc (-143.9987:-57.6:0.5);
				\draw [blue](1.9679,-0.8222) arc (-57.6035:0:0.5);
				\draw [->,red](1.1,-0.5) -- (1.2,-0.6);
				\draw [red](1.2,-0.6) -- (1.3,-0.7);
				\draw [red, dotted](0.9,-0.3) -- (1.1,-0.5);
			\end{tikzpicture}
			\caption{Examples of Cases 1 and 3 in the proof of Proposition~\ref{MorphismGreenWalk}(a).} \label{MorphWalk13}
		\end{figure}
		
		Case 2: $\alpha \in Q^{-1}_1$ and $w_i$ does not end on a peak. In this case, $w_{i+1}$ is given by adding a hook to the end of $w_i$. We first add a formal inverse $\beta^{-1}$ to the end of $w_i$. It follows that the arrow $\beta$ is lies in the cycle $\mathfrak{C}_{v_i}$ and therefore $\widehat{e}(\beta^{-1})=y^{v_i}$, where $y^{v_i}$ is the predecessor to $x_i^{v_i}$. If the edge $y$ associated to $y^{v_i}$ is truncated then $w_{i+1}=w_i \beta^{-1}$, as illustrated in Figure~\ref{MorphWalk24}(a)(i). In this subcase, the first two steps along a clockwise Green walk from $x_i^{v_i}$ are $\overline{y^{v_i}}$ and $y^{v_i}$ respectively. Thus $x_{i+1}^{v_{i+1}} = y^{v_i}$ and hence, $\widehat{e}(\beta^{-1})=x_{i+1}^{v_{i+1}}$. So $w_{i+1}$ satisfies property (i) at the start of the proof.
		
		On the other hand, if $y$ is not truncated then $w_{i+1}=w_i \beta^{-1} w'$, where $w'=\gamma_1\ldots\gamma_r$ is a maximal direct string. It follows from the maximality of $w'$ that $\widehat{e}(\gamma_r)=z^u$, where $z^u$ is the predecessor to $\overline{y^{v_i}}$, as illustrated in Figure~\ref{MorphWalk24}(a)(ii). The first two steps along a clockwise Green walk from $x_i^{v_i}$ are $\overline{y^{v_i}}$ and $\overline{z^u}$. Thus $x_{i+1}^{v_{i+1}} = \overline{z^u}$ and hence, $\widehat{e}(\gamma_r)=\overline{x_{i+1}^{v_{i+1}}}$. So $w_{i+1}$ satisfies property (ii).
		
		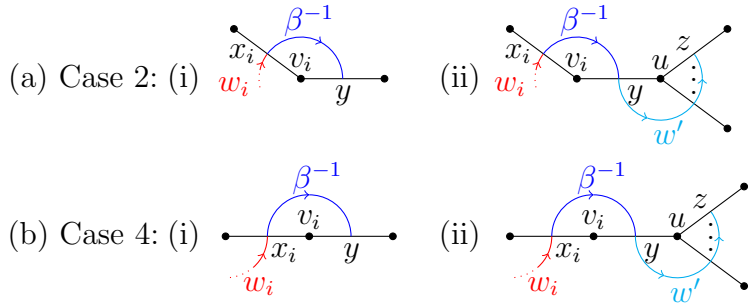
\begin{figure}[b]
			\centering
			\begin{tikzpicture}[scale=1.1]
				\draw (-3.6,1.5) node {(a) Case 2:};
				\draw (-2.4,1.5) node {(i)};
				\draw (0.9,1.5) node {(ii)};
				\draw (-1.8,2.1) -- (-1,1.5) -- (0,1.5);
				\draw (-1,1.7) node {$v_i$};
				\draw (-1.7,1.8) node {$x_i$};
				\draw (-0.491,1.3) node {$y$};
				\draw [red] (-1.8,1.4) node {$w_i$};
				\draw [blue] (-0.9,2.2) node {$\beta^{-1}$};
				\draw [fill=black] (0,1.5) ellipse (0.04 and 0.04);
				\draw [fill=black] (-1,1.5) ellipse (0.04 and 0.04);
				\draw [fill=black] (-1.8,2.1) ellipse (0.04 and 0.04);
				\draw [->,blue](-1.4,1.8) arc (143.1301:60:0.5);
				\draw [blue](-0.75,1.933) arc (59.9993:0:0.5);
				\draw [red](-1.4045,1.7939) arc (143.9987:156:0.5);
				\draw [->,red](-1.4891,1.604) arc (167.9957:156:0.5);
				\draw [red, dotted](-1.4891,1.396) arc (-167.9957:-192:0.5);

				\draw (1.5,2.1) -- (2.3,1.5) -- (3.3,1.5) -- (4.1,2.1);
				\draw (4.1,0.9) -- (3.3,1.5);
				\draw (3.7,1.6) node {$\vdots$};
				\draw (2.3,1.7) node {$v_i$};
				\draw (3.3,1.7) node {$u$};
				\draw (1.6,1.8) node {$x_i$};
				\draw (3,1.3) node {$y$};
				\draw (3.5727,1.9273) node {$z$};
				\draw [red] (1.5,1.4) node {$w_i$};
				\draw [blue] (2.4,2.2) node {$\beta^{-1}$};
				\draw [cyan] (3.4,0.8546) node {$w'$};
				\draw [fill=black] (4.1,2.1) ellipse (0.04 and 0.04);
				\draw [fill=black] (3.3,1.5) ellipse (0.04 and 0.04);
				\draw [fill=black] (4.1,0.9) ellipse (0.04 and 0.04);
				\draw [fill=black] (2.3,1.5) ellipse (0.04 and 0.04);
				\draw [fill=black] (1.5,2.1) ellipse (0.04 and 0.04);
				\draw [->,blue](1.9,1.8) arc (143.1301:60:0.5);
				\draw [blue](2.55,1.933) arc (59.9993:0:0.5);
				\draw [->,cyan](2.8,1.5) arc (180:252:0.5);
				\draw [->,cyan](3.1455,1.0245) arc (-108.0001:0:0.5);
				\draw [cyan](3.8,1.5) arc (0:36:0.5);
				\draw [red, dotted](1.8109,1.396) arc (-167.9957:-192:0.5);
				\draw [->, red](1.8109,1.604) arc (167.9957:156:0.5);
				\draw [red](1.8432,1.7034) arc (155.998:144:0.5);

				\draw (-3.6,-0.4) node {(b) Case 4:};
				\draw (-2.4,-0.4) node {(i)};
				\draw (0.9,-0.4) node {(ii)};
				\draw (-1.9,-0.4) -- (-0.9,-0.4) -- (0.1,-0.4);
				\draw (-0.9,-0.2) node {$v_i$};
				\draw (-1.2,-0.6) node {$x_i$};
				\draw (-0.391,-0.6) node {$y$};
				\draw [red] (-1.5,-1) node {$w_i$};
				\draw [blue] (-0.8,0.3) node {$\beta^{-1}$};
				\draw [fill=black] (0.1,-0.4) ellipse (0.04 and 0.04);
				\draw [fill=black] (-0.9,-0.4) ellipse (0.04 and 0.04);
				\draw [fill=black] (-1.9,-0.4) ellipse (0.04 and 0.04);
				\draw [->,blue](-1.4,-0.4) arc (180:90:0.5);
				\draw [blue](-0.9,0.1) arc (90.0002:0:0.5);
				\draw [red, ->](-1.5654,-0.7716) arc (-47.9992:-24:0.5);
				\draw [red](-1.4432,-0.6034) arc (-24.002:0:0.5);
				\draw [red,dotted](-1.8477,-0.8973) arc (-83.9964:-48:0.5);

				\draw (1.5,-0.4) -- (2.5,-0.4) -- (3.5,-0.4) -- (4.3,0.2);
				\draw (4.3,-1) -- (3.5,-0.4);
				\draw (3.9,-0.3) node {$\vdots$};
				\draw (2.5,-0.2) node {$v_i$};
				\draw (3.5,-0.2) node {$u$};
				\draw (2.2,-0.6) node {$x_i$};
				\draw (3.2,-0.6) node {$y$};
				\draw (3.7727,0.0273) node {$z$};
				\draw [red] (1.9,-1) node {$w_i$};
				\draw [blue] (2.6,0.3) node {$\beta^{-1}$};
				\draw [cyan] (3.6,-1.0454) node {$w'$};
				\draw [fill=black] (4.3,0.2) ellipse (0.04 and 0.04);
				\draw [fill=black] (3.5,-0.4) ellipse (0.04 and 0.04);
				\draw [fill=black] (4.3,-1) ellipse (0.04 and 0.04);
				\draw [fill=black] (2.5,-0.4) ellipse (0.04 and 0.04);
				\draw [fill=black] (1.5,-0.4) ellipse (0.04 and 0.04);
				\draw [->,blue](2,-0.4) arc (180:90:0.5);
				\draw [blue](2.5,0.1) arc (90.0002:0:0.5);
				\draw [->,cyan](3,-0.4) arc (-180:-108:0.5);
				\draw [->,cyan](3.3455,-0.8755) arc (-108.0001:0:0.5);
				\draw [cyan](4,-0.4) arc (0:36:0.5);
				\draw [red,dotted](1.5523,-0.8973) arc (-83.9964:-48:0.5);
				\draw [red,->](1.8346,-0.7716) arc (-47.9992:-24:0.5);
				\draw [red](1.9568,-0.6034) arc (-24.002:0:0.5);
			\end{tikzpicture}
			\caption{Examples of Cases 2 and 4 in the proof of Proposition~\ref{MorphismGreenWalk}(a).} \label{MorphWalk24}
		\end{figure}
		
		Case 3: $\alpha \in Q_1$ and $w_i$ ends on a peak. By Lemma~\ref{PeakDeep}(b)(i), $e(w)$ is truncated. Since $w_{i+1}=(w_i)_{-c}$ and there are no formal inverses at the end of $w_i$, it follows that $w_i= w_{i+1} \alpha$. In this case, $\widehat{s}(\alpha)=y^u$, where $y^u$ is the predecessor to $\overline{x_i^{v_i}}$, as illustrated in Figure~\ref{MorphWalk13}(b). The first two steps along a clockwise Green walk from $x_i^{v_i}$ are $\overline{x_i^{v_i}}$ and $\overline{y^u}$ respectively. Thus $x_{i+1}^{v_{i+1}} = \overline{y^u}$. If $w_{i+1}$ is a zero string, then $e(w_{i+1})=x_{i+1}$ and hence, $w_{i+1}=\varepsilon_{x_{i+1}}$. So $w_{i+1}$ satisfies property (iii) at the start of the proof. Otherwise, let $\gamma$ be the last symbol of $w_{i+1}$. If $\gamma \in Q^{-1}_1$, then we necessarily have $\widehat{e}(\gamma)=\overline{y^u}=x_{i+1}^{v_{i+1}}$. So $w_{i+1}$ satisfies property (i). If $\gamma \in Q_1$, then $\widehat{e}(\gamma)=y^u=\overline{x_{i+1}^{v_{i+1}}}$ and hence, $w_{i+1}$ satisfies property (ii).
		
		Case 4: $\alpha \in Q_1$ and $w_i$ does not end on a peak. The string combinatorics, illustrated in Figure~\ref{MorphWalk24}(b), is similar to Case 2. The clockwise Green walk from $x_i^{v_i}$ is also identical, so the result follows for this case by similar arguments.
		
		Case 5: $w_i$ is the zero string $\varepsilon_{x_i}$. This is only possible if either $i=0$ or $w_i = (w_{i-1})_{-c}$. Assume $0<i<k$. So $w_{i-1}=(w_i)_c= \gamma w'$ and $w_{i+1}=(w_i)_h=\beta^{-1}w''$ for some maximal inverse string $w'$ and some maximal direct string $w''$. It follows from the string combinatorics detailed in Cases 1 and 3, that $\widehat{s}(\gamma)=\overline{x_i^{v_i}}$. It then follows from Lemma~\ref{SimpleRay} that $\widehat{s}(\beta^{-1})=x_i^{v_i}$. The string combinatorics of adding a hook starting with $\beta^{-1}$ and the clockwise Green walk from $x_i^{v_i}$ is investigated in Case 2. Thus, the result for Case 5 follows by similar arguments to those used in Case 2. If $i=0$, then $w_1$ is defined in the proposition statement. The result then follows from setting $x_0^{v_0}=x^v$ and again using similar arguments to those in Case 2.
	
		The Proposition result follows from the inductive step outlined above, since setting $x_0^{v_0}=x^v$ in the cases where $w_0 = \varepsilon_x$ or $w_0$ ends with $\alpha_n \in Q^{-1}_1$, or setting $x_0^{v_0}=\overline{x^v}$ in the case where $w_0$ ends with $\alpha_n \in Q_1$ satisfies properties (i)-(iii) at the start of the proof.
		
		(b) The proof is similar to (a).
	\end{proof}
	
	Recall that the stable Auslander-Reiten quiver of a Brauer tree algebra is a finite tube (see for example \cite{GroupsWOGroups}). We will first distinguish between certain edges in the Brauer graph by introducing the notion of \emph{exceptional subtrees} of a graph. The motivation behind this definition lies in the fact that these particular subtrees of the Brauer graph have the same local structure of some Brauer tree, and hence, the string combinatorics along these subtrees behave in a similar manner to a Brauer tree algebra with exceptional multiplicity. Since the sequence (\ref{StrSeq}) defined at the start of Subsection~\ref{LocSimpleProj} terminates (in both directions) in a Brauer tree algebra, one might expect the sequence (\ref{StrSeq}) to terminate (in one direction) for certain modules related to the edges of the exceptional subtrees. This is indeed the case, and we will later show that the simple modules and the radicals of the indecomposable projectives associated to the edges of these subtrees belong to a tube.
	
	\begin{defn} \label{TreeDef}
		Let $G$ be a Brauer graph that is not a Brauer tree. Consider a subgraph $T$ of $G$ satisfying the following properties:
		\begin{enumerate}[label=(\roman*)]
			\item $T$ is a tree,
			\item $T$ has a unique vertex $v$ such that the graph $(G \setminus T) \cup \{v\}$ is connected,
			\item $T$ shares no vertex with any simple cycle of $G$, except at perhaps $v$,
			\item every vertex of $T$ has multiplicity 1, except for perhaps $v$.
		\end{enumerate}
		We will call such a subgraph an \emph{exceptional subtree} of $G$ and the vertex $v$ the \emph{connecting vertex} of $T$.
	\end{defn}
	
	Given a graph with exceptional subtrees, we can partition the edges of the graph into two distinct classes.
	
	\begin{defn}
		An edge of a Brauer graph $G$ is called an \emph{exceptional edge} if it belongs to some exceptional subtree of $G$. An edge of $G$ is otherwise called a \emph{non-exceptional edge}.
	\end{defn}
	
	Examples of exceptional subtrees are given in Figure~\ref{ExceptionalTrees}. The coloured edges of Figure~\ref{ExceptionalTrees} are the exceptional edges of $G$. All others are non-exceptional.	
	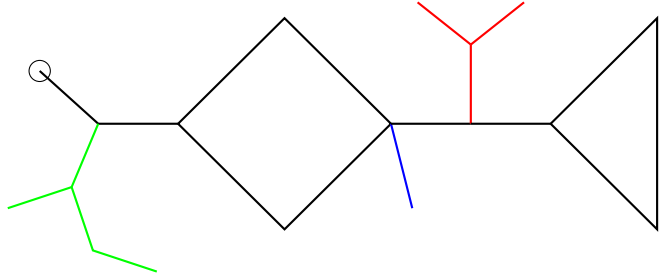
\begin{figure}[h]
		\centering
		\begin{tikzpicture}[scale=0.7]
			\draw [thick](0,0) -- (2,2) -- (4,0) -- (2,-2) -- cycle;
			\draw [thick,blue](4,0) -- (4.4,-1.6);
			\draw [thick](4,0) -- (5.5,0);
			\draw [thick][red](5.5,0) -- (5.5,1.5);
			\draw [thick][red](5.5,1.5) -- (4.5,2.3);
			\draw [thick][red](5.5,1.5) -- (6.5,2.3);
			\draw [thick](5.5,0) -- (7,0);
			\draw [thick](7,0) -- (9,2) -- (9,0) -- (9,-2) -- cycle;
			\draw [thick](0,0) -- (-1.5,0)  -- (-2.6,1);
			\draw (-2.6,1) ellipse (0.2 and 0.2);
			\draw [thick, green](-1.5,0) -- (-2,-1.2) -- (-3.2,-1.6);
			\draw [thick, green](-2,-1.2) -- (-1.6,-2.4) -- (-0.4,-2.8);
		\end{tikzpicture}
		\caption{Three distinct exceptional subtrees of a Brauer graph, coloured red, green and blue respectively. A vertex represented by a circle has multiplicity $\mathfrak{e}>1$.} \label{ExceptionalTrees}
	\end{figure}
	
	\begin{rem} \label{SubSubtree}
		An exceptional subtree of a Brauer graph can also contain as subgraphs further exceptional subtrees in the following sense. Suppose $x$ is an edge that belongs to an exceptional subtree $T$ of $G$. Consider the maximal subtree $T'$ of $T$ connected to (and possibly containing) $x$ via a vertex $v$ such that $T'$ shares no vertex with any non-exceptional edge. Then $T'$ satisfies properties (i)-(iv) of Definition~\ref{TreeDef} and is hence exceptional with connecting vertex $v$.
	\end{rem}
	
	There is a simple characterisation of the non-exceptional edges of a Brauer graph, as shown by the following.
	
	\begin{lem} \label{NonExcepClass}
		Let $G$ be a Brauer graph and $x$ be an edge in $G$. Then $x$ is non-exceptional if and only if it belongs to:
		\begin{enumerate}[label=(\roman*)]
			\item a simple cycle, or
			\item a simple path between two vertices $u,v$ belonging to simple cycles of $G$, or
			\item a simple path between a vertex $u$ belonging to a simple cycle and a vertex $v$ with $\mathfrak{e}_v>1$, or
			\item a simple path between vertices $u,v$ with $\mathfrak{e}_u,\mathfrak{e}_v>1$.
		\end{enumerate}
	\end{lem}
	\begin{proof}
		If $x$ belongs to a simple cycle, then either $x$ is a loop or both its vertices belong to a simple cycle, and hence, $x$ cannot be exceptional. Suppose $x$ instead belongs to simple path 
		\begin{equation*}
			p:\xymatrix@1{u \ar@{-}[r] &\cdots \ar@{-}[r]^x & \cdots \ar@{-}[r] & v}
		\end{equation*}	
		in $G$, where $u$ and $v$ satisfy either of (ii)-(iv). Suppose for a contradiction that there exists an exceptional subtree $T$ of $G$ containing $x$. Then $T$ has at most one vertex $v'$ that belongs to a simple cycle or has multiplicity $\mathfrak{e}_{v'}>1$. Thus, $T$ cannot contain both the vertices $u$ and $v$ of $p$. Consider the subgraph $G'=(G \setminus T) \cup \{v'\}$. It follows that $G'$ is not connected, since $T$ would otherwise contain multiple vertices that belong to a simple cycle of $G$, which would result from a simple path which avoids the edge $x$ from $p$. So no such tree $T$ exists and $x$ is non-exceptional.
		
		For the converse argument, we first assume $G$ is not a Brauer tree, and hence, that there exist edges in $G$ satisfying (i)-(iv). Suppose $x$ is an edge of $G$ that does not belong to any of (i)-(iv). Then at least one vertex connected to $x$ does not belong to a simple cycle and does not have multiplicity $\mathfrak{e}>1$. If both vertices satisfy this condition, we note that since $G$ is connected and is not a Brauer tree, there must exist a simple path
		\begin{equation*}
			q: \xymatrix@1{u \ar@{-}[r] & \cdots \ar@{-}[r] & u'},
		\end{equation*}
		where $u$ is a vertex connected to $x$ and $u'$ is a vertex that belongs to a simple cycle of $G$ or is such that $\mathfrak{e}_{u'}>1$. Let $v$ be the other vertex connected to $x$. Since $x$ does not satisfy (ii)-(iv), every path of source $v$ in $G$ not containing $x$ has no vertex belonging to a simple cycle and has no vertex of multiplicity $\mathfrak{e}>1$. Since $G$ is finite, the subgraph generated by all such paths of source $v$ is a tree $T$, which is an exceptional subtree of $G$. It follows that the subgraph $T'=T \cup \{x\}\cup\{u\}$ of $G$ is also an exceptional subtree of $G$ and $u$ is the connecting vertex of $T'$. Hence, $x$ is exceptional.
	\end{proof}
	
	\begin{rem}\label{NonConnect}
		It follows from the above that non-exceptional edges in a graph are all connected to each other. Specifically, if $G$ contains at least two non-exceptional edges and $x$ is a non-exceptional edge of $G$, then there exists a non-exceptional edge $y$ connected to $x$ via a common vertex. It also follows from the above that a non-exceptional edge is never truncated.
	\end{rem}

	We will now need to introduce some technical lemmata, which will be used extensively in the proofs of the main theorems. The first lemma, given below, is used primarily to construct direct or inverse strings through exceptional edges.

	\begin{lem} \label{MaxExceptional}
		Let $A=KQ/I$ be a representation-infinite Brauer graph algebra associated to a Brauer graph $G$. Let $v$ be a vertex in $G$ and $\alpha_1\ldots \alpha_m\beta_1\ldots\beta_n$ be the cycle $\mathfrak{C}_{v,\alpha_1}$ in $Q$, where $s(\alpha_1)=x$ and $s(\beta_1)=y$ for some edges $x$ and $y$ in $G$ such that $x \neq y$.
			
		\begin{enumerate}[label=(\alph*)]
			\item Suppose that $s(\beta_1), \ldots, s(\beta_n)$ each belong to an exceptional subtree of $G$ with connecting vertex $v$ and let $w_0$ be a string such that $e(w_0)=x$ and $w_0\beta^{-1}_n\ldots\beta^{-1}_1$ is a string. Then there exists a ray
			\begin{equation*}
				M(w_0) \rightarrow M(w_1) \rightarrow \cdots \rightarrow M(w_k) \rightarrow \cdots
			\end{equation*}
			in $_s\Gamma_A$ such that $w_k = w_0\beta^{-1}_n\ldots\beta^{-1}_1$ and $|w_i| > |w_0|$ for all $0<i \leq k$.
			
			\item Suppose that $e(\alpha_1), \ldots, e(\alpha_n)$ each belong to an exceptional subtree of $G$ with connecting vertex $v$ and let $w_0$ be a string such that $e(w_0)=x$ and $w_0\alpha_1\ldots\alpha_n$ is a string. Then there exists a ray
			\begin{equation*}
				\cdots \rightarrow M(w_{-k}) \rightarrow \cdots \rightarrow M(w_{-1}) \rightarrow M(w_0)
			\end{equation*}
			in $_s\Gamma_A$ such that $w_{-k} = w_0\alpha_1\ldots\alpha_n$ and $|w_{-i}| > |w_0|$ for all $0<i \leq k$.
		\end{enumerate}
	\end{lem}
	\begin{proof}
		(a) Let $(w_j)$ be the sequence in (\ref{StrSeq}). If $w_0$ is a zero string, then let $w_1=\gamma^{-1}\delta_1 \ldots \delta_r$ such that $\widehat{s}(\gamma^{-1})=x^{v}$. Otherwise, let $\alpha$ be the last symbol of $w_0$. Since $w_0\beta^{-1}_n$ is a string, we may conclude that if $\alpha \in Q^{-1}_1$ then $\widehat{e}(\alpha)=x^v$ and if $\alpha \in Q_1$ then $\widehat{e}(\alpha)=\overline{x^v}$. Thus by Proposition~\ref{MorphismGreenWalk}, $e(w_i)$ is determined by the $i$-th step along a double-stepped clockwise Green walk from $x^v$.
		
		Let $T$ be the exceptional subtree of $G$ with connecting vertex $v$ such that $T$ contains the edges $y_1=s(\beta_1), \ldots, y_n=s(\beta_n)$. The first $k$ steps along a clockwise double-stepped Green walk from $x^v$ step along the half-edges of $T$ until we reach the $(k+1)$-th step where we then reach a half-edge not in $T$. So by Proposition~\ref{MorphismGreenWalk}, $e(w_i)$ belongs to $T$ for all $i\leq k$. At each step, $w_i$ is of the form $w_0 \beta^{-1}_n w'_i$, where $w'_i$ is a string such that each symbol starts and ends at an edge in $T$. It therefore follows that $|w_i|>|w_0|$ for all $0<i\leq k$.
		
		We further note that a clockwise Green walk along a tree steps on both half-edges associated to each edge in the tree, so a clockwise double-stepped Green walk steps on precisely one half-edge for each edge in the tree (until we step on a half-edge not in $T$). One can show that the clockwise double-stepped Green walk from $x^v$ steps along the half-edges $y_n^{v}, \ldots, y_1^{v}$. In particular, one can show that $y_1^{v}$ is the $k$-th step along such a walk. Thus, $e(w_k)=y_1=y$.
		
		Since $T$ is a tree and all vertices of $T$ (except perhaps $v$) are of multiplicity 1, it is impossible to construct a string ending at $y_1$ that contains arrows or formal inverses around any vertex of $T$ other than $v$. Moreover, since each $w_i$ is determined by adding hooks or deleting cohooks from the end of the string and $e(w_i)$ belongs to $T$ for all $i \leq k$, we conclude that $w_k = w_0\beta^{-1}_n\ldots\beta^{-1}_1$.
		
		(b) The proof is similar to (a).
	\end{proof}
	
	In what follows, it is helpful to define a new form of Green walk, called a \emph{non-exceptional Green walk}. This is a Green walk that ignores exceptional edges.
	
	\begin{defn}
		By a \emph{non-exceptional Green walk} from a non-exceptional edge $x_0$ via a vertex $v_0$, we mean a sequence of half-edges $(x_j^{v_j})_{j \in \mathbb{Z}_{\geq 0}}$, where $x_{i+1}$ is connected to $x_i$ via the vertex $v_i$ and $\overline{x_{i+1}^{v_{i+1}}}$ is the first half-edge in the successor sequence of $x_i^{v_i}$ such that $x_{i+1}$ is non-exceptional. By a \emph{non-exceptional clockwise Green walk} from a non-exceptional edge $x_0$ via $v_0$, we mean a similar sequence $(x_j^{v_j})_{j\in\mathbb{Z}_{\geq 0}}$ of half-edges where each $\overline{x_{i+1}^{v_{i+1}}}$ is the first half-edge in the predecessor sequence of $x_i^{v_i}$ such that $x_{i+1}$ is non-exceptional.
	\end{defn}
	
	The next lemma shows that by skipping certain modules along a ray of source or target a string module $M(w_0)$, one can ignore the effect of exceptional edges when adding hooks or cohooks.
	
	\begin{lem} \label{PathWalk}
		Let $A=KQ/I$ be a representation-infinite Brauer graph algebra constructed from a Brauer graph $G$. Suppose $x_1$ and $x_2$ are non-exceptional edges incident to a vertex $v_1$ and let $\mathfrak{C}_{v_1,\alpha_1}=\alpha_1\ldots\alpha_m\beta_1\ldots\beta_n$, where $\widehat{s}(\alpha_1)=x_1^{v_1}$ and $\widehat{s}(\beta_1)=x_2^{v_1}$. Suppose $w_0$ is a string such that $e(w_0)=x_1$.
		\begin{enumerate}[label=(\alph*)]
			\item Suppose $x_2^{v_2}$ and $x_3^{v_3}$ are the first and second steps along a non-exceptional clockwise Green walk from $x_1^{v_1}$ respectively and suppose $w_0 \beta^{-1}_n\ldots\beta^{-1}_1$ is a string. Then there exists a ray
			\begin{equation*}
				M(w_0) \rightarrow M(w_1) \rightarrow \cdots \rightarrow M(w_k) \rightarrow \cdots
			\end{equation*}
			in $_s\Gamma_A$ such that $w_k=w_0 \beta^{-1}_n\ldots\beta^{-1}_1 w'$, where $w'=\gamma_1\ldots\gamma_r$ is the direct string of greatest length such that $\widehat{e}(\gamma_r)=\overline{x_3^{v_3}}$. Furthermore, $e(w_i)$ is exceptional for all $0<i < k$ and $|w_i| > |w_0|$ for all $0<i \leq k$.
			
			\item Suppose $x_2^{v_2}$ and $x_3^{v_3}$ are the first and second steps along a non-exceptional Green walk from $x_1^{v_1}$ respectively and suppose $w_0 \alpha_1\ldots\alpha_m$ is a string. Then there exists a ray
			\begin{equation*}
				\cdots \rightarrow M(w_{-k}) \rightarrow \cdots \rightarrow M(w_{-1}) \rightarrow M(w_0)
			\end{equation*}
			in $_s\Gamma_A$ such that $w_{-k}=w_0 \alpha_1\ldots\alpha_m w'$, where $w'=\gamma^{-1}_1\ldots\gamma^{-1}_r$ is the inverse string of greatest length such that $\widehat{e}(\gamma^{-1}_r)=\overline{x_3^{v_3}}$. Furthermore, $e(w_{-i})$ is exceptional for all $0<i < k$ and $|w_{-i}| > |w_0|$ for all $0<i \leq k$.
		\end{enumerate}
	\end{lem}
	\begin{proof}
		(a) Let $(w_j)$ be the sequence in (\ref{StrSeq}) with $w_0$ as defined in the lemma. If $w_0$ is a zero string, then assume the first symbol $\alpha^{-1}$ of $w_1$ is such that $\widehat{s}(\alpha^{-1})=x_1^{v_1}$. Since $x_2$ is given by a non-exceptional clockwise Green walk from $x_1$ via $v_1$, it follows that $s(\beta_i)$ is exceptional for all $i>1$. Additionally, since $w_0 \beta^{-1}_n \ldots\beta^{-1}_{2}$ is a string, Lemma~\ref{MaxExceptional}(a) applies and there exists a ray
		\begin{equation*}
			M(w_0) \rightarrow M(w_1) \rightarrow \cdots \rightarrow M(w_{l}) \rightarrow \cdots
		\end{equation*}
		in $_s\Gamma_A$ such that $w_l = w_0\beta^{-1}_n\ldots\beta^{-1}_{2}$ and $|w_i| > |w_0|$ for all $0<i \leq l$. It follows from Proposition~\ref{MorphismGreenWalk}(a) and the proof of Lemma~\ref{MaxExceptional}(a) that $e(w_i)$ is exceptional for all $0<i \leq l$.
			
		Since $w_l \beta^{-1}_1$ is a string, $w_l$ does not end on a peak and hence, $w_{l+1}=(w_l)_h=w_l\beta^{-1}_1 w''$, where $w''=\gamma_1\ldots\gamma_t$ is a maximal direct string. So $|w_{l+1}|>|w_0|$. If $e(w_{l+1})$ is non-exceptional, then $\widehat{e}(\gamma_t)=\overline{x_3^{v_3}}$, as required. So suppose instead $e(w_{l+1})$ is exceptional. Then there exists an integer $r$ such that $\widehat{e}(\gamma_r)=\overline{x_3^{v_3}}$ and $e(\gamma_i)$ is exceptional for all $r < i \leq t$. In particular, the string $w' =\gamma_1\ldots\gamma_r$ is the direct string of greatest length such that $\widehat{e}(\gamma_r)=\overline{x_3^{v_3}}$. By Lemma~\ref{MaxExceptional}(b), there exists a ray
		\begin{equation*}
			\cdots \rightarrow M(w_{l+1}) \rightarrow \cdots \rightarrow M(w_{k-1}) \rightarrow M(w_k)
		\end{equation*}
		in $_s\Gamma_A$ such that $w_k=w_l \beta^{-1}_1 w'$ and $w_{l+1}$ is as above. Moreover, $|w_i| > |w_k|>|w_0|$ for all $l+1 \leq i < k$. This ray belongs to the same ray containing $M(w_0)$ and $M(w_l)$, as all changes are made to the end of the string. The result then follows.
		
		(b) The proof is similar to (a).
	\end{proof}
	
	\begin{exam}
		Consider the following Brauer graph $G$, where the circled vertices $u_3$ and $u_4$ have a multiplicity of two and all other vertices have multiplicity one.
		
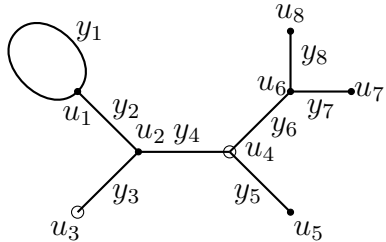
\begin{center}
			\begin{tikzpicture}[scale=0.8]
				\draw[thick] (1.5,0) -- (2.5,1) -- (2.5,2);
				\draw[thick] (3.5,1) -- (2.5,1);
				\draw[thick] (1.5,0) -- (2.5,-1);
				\draw[thick] (1.5,0) -- (0,0) -- (-1,-1);
				\draw[thick] (0,0) -- (-1,1);
				\draw[thick] (-1,1) .. controls (-0.5,1.5) and (-1.5,2.5) .. (-2,2) .. controls (-2.5,1.5) and (-1.5,0.5) .. (-1,1);
	
				\draw [fill=black] (-1,1) ellipse (0.05 and 0.05);
				\draw (-1,-1) ellipse (0.1 and 0.1);
				\draw [fill=black] (0,0) ellipse (0.05 and 0.05);
				\draw [fill=black] (2.5,-1) ellipse (0.05 and 0.05);
				\draw (1.5,0) ellipse (0.1 and 0.1);
				\draw [fill=black] (2.5,1) ellipse (0.05 and 0.05);
				\draw [fill=black] (3.5,1) ellipse (0.05 and 0.05);
				\draw [fill=black] (2.5,2) ellipse (0.05 and 0.05);
	
				\draw (-1,0.6) node {$u_1$};
				\draw (0.2,0.3) node {$u_2$};
				\draw (-1.2,-1.3) node {$u_3$};
				\draw (2,0) node {$u_4$};
				\draw (2.8,-1.3) node {$u_5$};
				\draw (2.2,1.1) node {$u_6$};
				\draw (3.8,1) node {$u_7$};
				\draw (2.5,2.3) node {$u_8$};
	
				\draw (-0.8,2) node {$y_1$};		
				\draw (-0.2,0.7) node {$y_2$};
				\draw (-0.2,-0.7) node {$y_3$};
				\draw (0.8,0.3) node {$y_4$};
				\draw (1.8,-0.7) node {$y_5$};
				\draw (2.4,0.4) node {$y_6$};
				\draw (3,0.7) node {$y_7$};
				\draw (2.9,1.6) node {$y_8$};
			\end{tikzpicture}
		\end{center}
		The exceptional edges of $G$ are $y_5$, $y_6$, $y_7$ and $y_8$. The non-exceptional edges of $G$ are $y_1$, $y_2$, $y_3$ and $y_4$.
		
		Let $\mathfrak {C}_{u_2,\beta_1}=\beta_1\beta_2\beta_3$ and $\mathfrak {C}_{u_4,\gamma_1}=\gamma_1\gamma_2\gamma_3$, where $s(\beta_1)=y_2$ and $s(\gamma_1)=y_4$. Consider the string $w_0= \beta_1^{-1}$. The first two steps along a non-exceptional clockwise Green walk from $y_2^{u_2}$ are $y_4^{u_4}$ and $y_4^{u_2}$ respectively. Since $w_0\beta^{-1}_3$ is a string, Lemma~\ref{PathWalk}(a) implies there exists a ray 
		\begin{equation*} 
			M(w_0) \rightarrow M(w_1) \rightarrow \cdots \rightarrow M(w_k) \rightarrow \cdots
		\end{equation*}
		in $_s\Gamma_A$ such that $w_k=w_0\beta^{-1}_3\gamma_1\gamma_2\gamma_3$. Moreover, $e(w_i)$ is exceptional for all $0<i < k$ and $|w_i| > |w_0|$ for all $0<i \leq k$. One can verify that $k=5$ in this example, as illustrated in Figure~\ref{PathWalkExample}.
		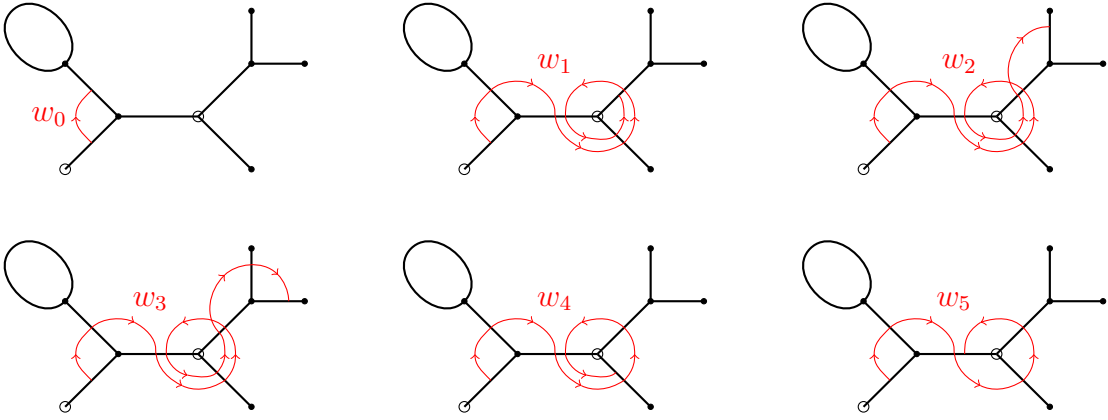
\begin{figure}[h]
			\centering
			\begin{tikzpicture}[scale=0.7]
				\draw[thick] (-6,4.5) -- (-5,5.5) -- (-5,6.5);
				\draw[thick] (-4,5.5) -- (-5,5.5);
				\draw[thick] (-6,4.5) -- (-5,3.5);
				\draw[thick] (-6,4.5) -- (-7.5,4.5) -- (-8.5,3.5);
				\draw[thick] (-7.5,4.5) -- (-8.5,5.5);
				\draw[thick] (-8.5,5.5) .. controls (-8,6) and (-9,7) .. (-9.5,6.5) .. controls (-10,6) and (-9,5) .. (-8.5,5.5);
	
				\draw [fill=black] (-8.5,5.5) ellipse (0.05 and 0.05);
				\draw (-8.5,3.5) ellipse (0.1 and 0.1);
				\draw [fill=black] (-7.5,4.5) ellipse (0.05 and 0.05);
				\draw [fill=black] (-5,3.5) ellipse (0.05 and 0.05);
				\draw (-6,4.5) ellipse (0.1 and 0.1);
				\draw [fill=black] (-5,5.5) ellipse (0.05 and 0.05);
				\draw [fill=black] (-4,5.5) ellipse (0.05 and 0.05);
				\draw [fill=black] (-5,6.5) ellipse (0.05 and 0.05);
	
				\draw [red, ->](-8,4) .. controls (-8.2,4.2) and (-8.3,4.3) .. (-8.3,4.5);
				\draw [red](-8.3,4.5) .. controls (-8.3,4.7) and (-8.2,4.8) .. (-8,5);
	
				\draw[thick] (1.5,4.5) -- (2.5,5.5) -- (2.5,6.5);
				\draw[thick] (3.5,5.5) -- (2.5,5.5);
				\draw[thick] (1.5,4.5) -- (2.5,3.5);
				\draw[thick] (1.5,4.5) -- (0,4.5) -- (-1,3.5);
				\draw[thick] (0,4.5) -- (-1,5.5);
				\draw[thick] (-1,5.5) .. controls (-0.5,6) and (-1.5,7) .. (-2,6.5) .. controls (-2.5,6) and (-1.5,5) .. (-1,5.5);
	
				\draw [fill=black] (-1,5.5) ellipse (0.05 and 0.05);
				\draw (-1,3.5) ellipse (0.1 and 0.1);
				\draw [fill=black] (0,4.5) ellipse (0.05 and 0.05);
				\draw [fill=black] (2.5,3.5) ellipse (0.05 and 0.05);
				\draw (1.5,4.5) ellipse (0.1 and 0.1);
				\draw [fill=black] (2.5,5.5) ellipse (0.05 and 0.05);
				\draw [fill=black] (3.5,5.5) ellipse (0.05 and 0.05);
				\draw [fill=black] (2.5,6.5) ellipse (0.05 and 0.05);
	
				\draw [red, ->](-0.5,4) .. controls (-0.7,4.2) and (-0.8,4.3) .. (-0.8,4.5);
				\draw [red](-0.8,4.5) .. controls (-0.8,4.7) and (-0.7,4.8) .. (-0.5,5);
				\draw [red,->](-0.5,5) .. controls (-0.3,5.2) and (0.1,5.2) .. (0.3,5.1);
				\draw [red,->](0.3,5.1) .. controls (0.5,5) and (0.7,4.8) .. (0.7,4.6) .. controls (0.7,4.2) and (1,4) .. (1.2,3.9);
				\draw [red,->](1.2,3.9) .. controls (1.4,3.8) and (1.7,3.8) .. (2,4) .. controls (2.2,4.2) and (2.2,4.4) .. (2.2,4.5);
				\draw [red,->](2.2,4.5) .. controls (2.2,4.6) and (2.2,4.8) .. (2,5) .. controls (1.8,5.2) and (1.5,5.2) .. (1.2,5.1);
				\draw [red,->](1.2,5.1) .. controls (0.9,4.9) and (0.9,4.7) .. (0.9,4.5) .. controls (0.9,4.4) and (1.05,4.15) .. (1.3,4.1);
				\draw [red,->](1.3,4.1) .. controls (1.55,4.05) and (1.75,4.05) .. (1.85,4.15) .. controls (1.95,4.25) and (2,4.4) .. (2,4.5);
				\draw [red](2,4.5) .. controls (2,4.6) and (2,4.8) .. (1.9,4.9);
	
				\draw[thick] (9,4.5) -- (10,5.5) -- (10,6.5);			
				\draw[thick] (11,5.5) -- (10,5.5);
				\draw[thick] (9,4.5) -- (10,3.5);
				\draw[thick] (9,4.5) -- (7.5,4.5) -- (6.5,3.5);
				\draw[thick] (7.5,4.5) -- (6.5,5.5);
				\draw[thick] (6.5,5.5) .. controls (7,6) and (6,7) .. (5.5,6.5) .. controls (5,6) and (6,5) .. (6.5,5.5);
	
				\draw [fill=black] (6.5,5.5) ellipse (0.05 and 0.05);
				\draw (6.5,3.5) ellipse (0.1 and 0.1);
				\draw [fill=black] (7.5,4.5) ellipse (0.05 and 0.05);
				\draw [fill=black] (10,3.5) ellipse (0.05 and 0.05);
				\draw (9,4.5) ellipse (0.1 and 0.1);
				\draw [fill=black] (10,5.5) ellipse (0.05 and 0.05);
				\draw [fill=black] (11,5.5) ellipse (0.05 and 0.05);
				\draw [fill=black] (10,6.5) ellipse (0.05 and 0.05);
	
				\draw [red, ->](7,4) .. controls (6.8,4.2) and (6.7,4.3) .. (6.7,4.5);
				\draw [red](6.7,4.5) .. controls (6.7,4.7) and (6.8,4.8) .. (7,5);
				\draw [red,->](7,5) .. controls (7.2,5.2) and (7.6,5.2) .. (7.8,5.1);
				\draw [red,->](7.8,5.1) .. controls (8,5) and (8.2,4.8) .. (8.2,4.6) .. controls (8.2,4.2) and (8.5,4) .. (8.7,3.9);
				\draw [red,->](8.7,3.9) .. controls (8.9,3.8) and (9.2,3.8) .. (9.5,4) .. controls (9.7,4.2) and (9.7,4.4) .. (9.7,4.5);
				\draw [red,->](9.7,4.5) .. controls (9.7,4.6) and (9.7,4.8) .. (9.5,5) .. controls (9.3,5.2) and (9,5.2) .. (8.7,5.1);
				\draw [red,->](8.7,5.1) .. controls (8.4,4.9) and (8.4,4.7) .. (8.4,4.5) .. controls (8.4,4.4) and (8.55,4.15) .. (8.8,4.1);
				\draw [red,->](8.8,4.1) .. controls (9.05,4.05) and (9.25,4.05) .. (9.35,4.15) .. controls (9.45,4.25) and (9.5,4.4) .. (9.5,4.5);
				\draw [red,->](9.5,4.5) .. controls (9.5,4.6) and (9.5,4.8) .. (9.4,4.9) .. controls (9.1,5.2) and (9.2,5.7) .. (9.5,6);
				\draw [red](9.5,6) .. controls (9.7,6.2) and (9.9,6.2) .. (10,6.2);
	
				\draw[thick] (-6,0) -- (-5,1) -- (-5,2);
				\draw[thick] (-4,1) -- (-5,1);
				\draw[thick] (-6,0) -- (-5,-1);
				\draw[thick] (-6,0) -- (-7.5,0) -- (-8.5,-1);
				\draw[thick] (-7.5,0) -- (-8.5,1);
				\draw[thick] (-8.5,1) .. controls (-8,1.5) and (-9,2.5) .. (-9.5,2) .. controls (-10,1.5) and (-9,0.5) .. (-8.5,1);
	
				\draw [fill=black] (-8.5,1) ellipse (0.05 and 0.05);
				\draw (-8.5,-1) ellipse (0.1 and 0.1);
				\draw [fill=black] (-7.5,0) ellipse (0.05 and 0.05);
				\draw [fill=black] (-5,-1) ellipse (0.05 and 0.05);
				\draw (-6,0) ellipse (0.1 and 0.1);
				\draw [fill=black] (-5,1) ellipse (0.05 and 0.05);
				\draw [fill=black] (-4,1) ellipse (0.05 and 0.05);
				\draw [fill=black] (-5,2) ellipse (0.05 and 0.05);
	
				\draw [red, ->](-8,-0.5) .. controls (-8.2,-0.3) and (-8.3,-0.2) .. (-8.3,0);
				\draw [red](-8.3,0) .. controls (-8.3,0.2) and (-8.2,0.3) .. (-8,0.5);
				\draw [red,->](-8,0.5) .. controls (-7.8,0.7) and (-7.4,0.7) .. (-7.2,0.6);
				\draw [red,->](-7.2,0.6) .. controls (-7,0.5) and (-6.8,0.3) .. (-6.8,0.1) .. controls (-6.8,-0.3) and (-6.5,-0.5) .. (-6.3,-0.6);
				\draw [red,->](-6.3,-0.6) .. controls (-6.1,-0.7) and (-5.8,-0.7) .. (-5.5,-0.5) .. controls (-5.3,-0.3) and (-5.3,-0.1) .. (-5.3,0);
				\draw [red,->](-5.3,0) .. controls (-5.3,0.1) and (-5.3,0.3) .. (-5.5,0.5) .. controls (-5.7,0.7) and (-6,0.7) .. (-6.3,0.6);
				\draw [red,->](-6.3,0.6) .. controls (-6.6,0.4) and (-6.6,0.2) .. (-6.6,0) .. controls (-6.6,-0.1) and (-6.45,-0.35) .. (-6.2,-0.4);
				\draw [red,->](-6.2,-0.4) .. controls (-5.95,-0.45) and (-5.75,-0.45) .. (-5.65,-0.35) .. controls (-5.55,-0.25) and (-5.5,-0.1) .. (-5.5,0);
				\draw [red,->](-5.5,0) .. controls (-5.5,0.1) and (-5.5,0.3) .. (-5.6,0.4) .. controls (-5.9,0.7) and (-5.8,1.2) .. (-5.5,1.5);
				\draw [red,->](-5.5,1.5) .. controls (-5.3,1.7) and (-4.8,1.8) .. (-4.5,1.5);
				\draw [red](-4.5,1.5) .. controls (-4.4,1.4) and (-4.3,1.2) .. (-4.3,1);
	
				\draw[thick] (1.5,0) -- (2.5,1) -- (2.5,2);			
				\draw[thick] (3.5,1) -- (2.5,1);
				\draw[thick] (1.5,0) -- (2.5,-1);
				\draw[thick] (1.5,0) -- (0,0) -- (-1,-1);
				\draw[thick] (0,0) -- (-1,1);
				\draw[thick] (-1,1) .. controls (-0.5,1.5) and (-1.5,2.5) .. (-2,2) .. controls (-2.5,1.5) and (-1.5,0.5) .. (-1,1);
	
				\draw [fill=black] (-1,1) ellipse (0.05 and 0.05);
				\draw (-1,-1) ellipse (0.1 and 0.1);
				\draw [fill=black] (0,0) ellipse (0.05 and 0.05);
				\draw [fill=black] (2.5,-1) ellipse (0.05 and 0.05);
				\draw (1.5,0) ellipse (0.1 and 0.1);
				\draw [fill=black] (2.5,1) ellipse (0.05 and 0.05);
				\draw [fill=black] (3.5,1) ellipse (0.05 and 0.05);
				\draw [fill=black] (2.5,2) ellipse (0.05 and 0.05);
	
				\draw [red, ->](-0.5,-0.5) .. controls (-0.7,-0.3) and (-0.8,-0.2) .. (-0.8,0);
				\draw [red](-0.8,0) .. controls (-0.8,0.2) and (-0.7,0.3) .. (-0.5,0.5);
				\draw [red,->](-0.5,0.5) .. controls (-0.3,0.7) and (0.1,0.7) .. (0.3,0.6);
				\draw [red,->](0.3,0.6) .. controls (0.5,0.5) and (0.7,0.3) .. (0.7,0.1) .. controls (0.7,-0.3) and (1,-0.5) .. (1.2,-0.6);
				\draw [red,->](1.2,-0.6) .. controls (1.4,-0.7) and (1.7,-0.7) .. (2,-0.5) .. controls (2.2,-0.3) and (2.2,-0.1) .. (2.2,0);
				\draw [red,->](2.2,0) .. controls (2.2,0.1) and (2.2,0.3) .. (2,0.5) .. controls (1.8,0.7) and (1.5,0.7) .. (1.2,0.6);
				\draw [red,->](1.2,0.6) .. controls (0.9,0.4) and (0.9,0.2) .. (0.9,0) .. controls (0.9,-0.1) and (1.05,-0.35) .. (1.3,-0.4);
				\draw [red](1.3,-0.4) .. controls (1.55,-0.45) and (1.75,-0.45) .. (1.85,-0.35);
	
				\draw[thick] (9,0) -- (10,1) -- (10,2);
				\draw[thick] (11,1) -- (10,1);
				\draw[thick] (9,0) -- (10,-1);
				\draw[thick] (9,0) -- (7.5,0) -- (6.5,-1);
				\draw[thick] (7.5,0) -- (6.5,1);
				\draw[thick] (6.5,1) .. controls (7,1.5) and (6,2.5) .. (5.5,2) .. controls (5,1.5) and (6,0.5) .. (6.5,1);
	
				\draw [fill=black] (6.5,1) ellipse (0.05 and 0.05);
				\draw (6.5,-1) ellipse (0.1 and 0.1);
				\draw [fill=black] (7.5,0) ellipse (0.05 and 0.05);
				\draw [fill=black] (10,-1) ellipse (0.05 and 0.05);
				\draw (9,0) ellipse (0.1 and 0.1);
				\draw [fill=black] (10,1) ellipse (0.05 and 0.05);
				\draw [fill=black] (11,1) ellipse (0.05 and 0.05);
				\draw [fill=black] (10,2) ellipse (0.05 and 0.05);
	
				\draw [red, ->](7,-0.5) .. controls (6.8,-0.3) and (6.7,-0.2) .. (6.7,0);
				\draw [red](6.7,0) .. controls (6.7,0.2) and (6.8,0.3) .. (7,0.5);
				\draw [red,->](7,0.5) .. controls (7.2,0.7) and (7.6,0.7) .. (7.8,0.6);
				\draw [red,->](7.8,0.6) .. controls (8,0.5) and (8.2,0.3) .. (8.2,0.1) .. controls (8.2,-0.3) and (8.5,-0.5) .. (8.7,-0.6);
				\draw [red,->](8.7,-0.6) .. controls (8.9,-0.7) and (9.2,-0.7) .. (9.5,-0.5) .. controls (9.7,-0.3) and (9.7,-0.1) .. (9.7,0);
				\draw [red,->](9.7,0) .. controls (9.7,0.1) and (9.7,0.3) .. (9.5,0.5) .. controls (9.3,0.7) and (9,0.7) .. (8.7,0.6);
				\draw [red](8.7,0.6) .. controls (8.4,0.4) and (8.4,0.2) .. (8.4,0);
	
				\draw[red] (-8.8,4.5) node {$w_0$};
				\draw[red] (0.7,5.5) node {$w_1$};
				\draw[red] (8.3,5.5) node {$w_2$};
				\draw[red] (-6.9,1) node {$w_3$};
				\draw[red] (0.7,1) node {$w_4$};
				\draw[red] (8.2,1) node {$w_5$};
			\end{tikzpicture}
			\caption{The first 5 terms in a sequence of strings given by (\ref{StrSeq}). The multiplicity of the circled vertices is two and all other vertices have multiplicity one.} \label{PathWalkExample}
		\end{figure}
		
		One may also notice from Figure~\ref{PathWalkExample} the use of Lemma~\ref{MaxExceptional}(b) in the proof of Lemma~\ref{PathWalk}(a) on the ray segment
		\begin{equation*}
			M(w_1) \rightarrow \cdots \rightarrow M(w_5).
		\end{equation*}
	\end{exam}
	
	The following remark is useful for the next lemma.
	\begin{rem} \label{NonExceptionalCycle}
		Non-exceptional Green walks are periodic. Thus, one can perform a non-exceptional (clockwise or anticlockwise) Green walk to construct a (not necessarily simple) cycle
		\begin{equation*}
			c: \xymatrix{v_0 \ar@{-}[r]^-{x_1} & v_1 \ar@{-}[r]^-{x_2} & v_2 \ar@{-}[r] & \cdots \ar@{-}[r] &  v_{m-1} \ar@{-}[r]^-{x_m} & v_0},
		\end{equation*}
		of $G$ consisting of non-exceptional-edges such that $x_{i+1}^{v_i}$ is the first half-edge in the predecessor (resp. successor) sequence of $x_{i}^{v_i}$ such that $x_{i+1}$ is non-exceptional.
	\end{rem}
	
	\begin{lem}\label{NonRay}
		Let $A=KQ/I$ be a Brauer graph algebra associated to a Brauer graph $G$. Let $M(w_0)$ be the string module associated to a string $w_0=\alpha_1\ldots\alpha_n$.
		\begin{enumerate}[label=(\alph*)]
			\item \begin{enumerate}[label=(\roman*)]
				\item If $e(w_0)$ is a non-exceptional edge in $G$ and $\alpha_n\in Q_1$, then the ray in $_s\Gamma_A$ of source $M(w_0)$ given by adding or deleting from the end of $w_0$ is infinite. Furthermore, each module $M(w_i)$ along the ray is such that $|w_i|>|w_0|$ for all $i>0$.
				\item If $x$ is a non-exceptional edge in $G$ and $w_0=\varepsilon_x$, then both rays in $_s\Gamma_A$ of source $M(w_0)$ are infinite. Furthermore, $|w|>|w_0|$ for any module $M(w) \not\cong M(w_0)$ along any such ray.
			\end{enumerate}
			\item \begin{enumerate}[label=(\roman*)]
				\item If $e(w_0)$ is a non-exceptional edge in $G$ and $\alpha_n\in Q^{-1}_1$, then the ray in $_s\Gamma_A$ of target $M(w_0)$ given by adding or deleting from the end of $w_0$ is infinite. Furthermore, each module $M(w_{-i})$ along the ray is such that $|w_{-i}|>|w_0|$ for all $i>0$.
				\item If $x$ is a non-exceptional edge in $G$ and $w_0=\varepsilon_x$, then both rays in $_s\Gamma_A$ of target $M(w_0)$ are infinite. Furthermore, $|w|>|w_0|$ for any module $M(w) \not\cong M(w_0)$ along any such ray.
			\end{enumerate}
		\end{enumerate}
	\end{lem}
	
	\begin{proof}
		The proof relies upon the iterative use of Lemma~\ref{PathWalk}.
		
		(a)(i) Let $(w_j)$ be sequence (\ref{StrSeq}). Let $x_1$ be the non-exceptional edge such that $e(w_0)= x_1$ and let $x_1^{v_0}$ be a half-edge associated to $x_1$ such that $\widehat{e}(\alpha_n)=x_1^{v_0}$. We may perform a non-exceptional clockwise Green walk from $x_1^{v_1} = \overline{x_1^{v_0}}$ to construct a (not necessarily simple) cycle $c$ in $G$ of the form given in Remark~\ref{NonExceptionalCycle}.
		
		Since $\alpha_n \in Q_1$, there exists a non-zero inverse string $\beta^{-1}_1\ldots \beta^{-1}_r$ such that $\widehat{s}(\beta^{-1}_1)=x_1^{v_1}$, $\widehat{e}(\beta^{-1}_r)=x_2^{v_1}$ and $w_0 \beta^{-1}_1\ldots \beta^{-1}_r$ is a string. Thus by Lemma~\ref{PathWalk}(a), there exists a ray
		\begin{equation*}
			M(w_0) \rightarrow M(w_1) \rightarrow \cdots \rightarrow M(w_k) \rightarrow \cdots
		\end{equation*}
		in $_s\Gamma_A$ such that $e(w_k)=x_3$ and $|w_i| > |w_0|$ for all $0<i \leq k$. Furthermore, the last symbol $\alpha$ of $w_k$ is an arrow such that $\widehat{e}(\alpha)=x_3^{v_2}$.
		
		Since $w_k$ satisfies similar properties to $w_0$, we may use the above argument iteratively along the cycle $c$. Thus, the sequence (\ref{StrSeq}) never terminates, and hence, the ray of source $M(w_0)$ given by adding or deleting from the end of $w_0$ is infinite. Moreover, $|w_i| > |w_0|$ for all $i>0$.
		
		(a)(ii) Let $(w_j)$ be sequence (\ref{StrSeq}) and let $x^v$ be a half-edge associated to $x$. There are two rays of source $S(x)$ in $_s\Gamma_A$. These are obtained by choosing $w_1$ such that the first symbol $\beta^{-1}$ of $w_1$ is a formal inverse with either $\widehat{s}(\beta^{-1})=x^v$ or $\widehat{s}(\beta^{-1})=\overline{x^v}$.
		
		A clockwise non-exceptional Green walk from either $x^v$ or $\overline{x^v}$ induces a cycle of non-exceptional edges similar to that in Remark~\ref{NonExceptionalCycle}. The conditions of Lemma~\ref{PathWalk}(a) are satisfied for any zero string associated to a non-exceptional edge, so similar arguments to (i) show that both possible sequences $(w_j)$ starting with $w_0$ are infinite and $|w_i| > |w_0|$ for all $i>0$. Thus, both rays in $_s\Gamma_A$ of source $M(w_0)$ are infinite and $|w|>|w_0|$ for any module $M(w) \not\cong M(w_0)$ along any such ray.
		
		(b) The proof of (b)(i) and (b)(ii) is similar to (a)(i) and (a)(ii) respectively.
	\end{proof}
	
	We now prove the main results of this section.
	
	\begin{thm} \label{EdgeTubes}
		Let $A$ be a representation-infinite Brauer graph algebra associated to a Brauer graph $G$ and let $x$ be an edge in $G$. Then the simple module $S(x)$ and the radical of the projective $P(x)$ belong to exceptional tubes of $_s\Gamma_A$ if and only if $x$ is an exceptional edge.
	\end{thm}
	\begin{proof}
		$(\Rightarrow:)$ Suppose $x$ is non-exceptional and consider the simple module $S(x)$. This is associated to the zero string $w=\varepsilon_x$. Thus, Lemma~\ref{NonRay}(a)(ii) applies and so both rays of source $S(x)$ in $_s\Gamma_A$ are infinite. Hence, $S(x)$ does not belong to a tube. 
		
		Now consider the module $\rad P(x)$ and instead let $w$ be the string such that $M(w)=P(x) / \soc P(x) = \tau^{-1} \rad P(x)$. We aim to show that $M(w)$ does not belong to a tube, since it then follows that $\rad P(x)$ does not belong to a tube. Note that $w=w' w''$, where $M(w'),M(w'')\in\mathcal{M}$. That is, $w'$ is a maximal inverse string and $w''$ is a maximal direct string such that $e(w')=x=s(w'')$.
		
		If $e(w)$ is non-exceptional, then we can apply Lemma~\ref{NonRay}(a)(i) to show that the ray of source $M(w)$ given by adding or deleting from the end of the string is infinite. Otherwise if $e(w)$ is exceptional, then suppose $w''=\beta_1\ldots\beta_n$ and let $r$ be the greatest integer such that $\widehat{e}(\beta_r)$ is the first half-edge associated to a non-exceptional edge in the predecessor sequence from $\widehat{e}(\beta_n)$. Then by Lemma~\ref{MaxExceptional}(b), there exists a ray
		\begin{equation*}
			\cdots \rightarrow M(w_{-k}) \rightarrow \cdots \rightarrow M(w_{-1}) \rightarrow M(w_0)
		\end{equation*}
		in $_s\Gamma_A$ such that $w_0=w'\beta_1\ldots \beta_r$ and $w_{-k}=w$. The string $w_0$ satisfies the conditions of Lemma~\ref{NonRay}(a)(i), and therefore the ray in $_s\Gamma_A$ of source $M(w_0)$ given by adding or deleting from the end of $w_0$ is infinite. This is contained within the ray of source $M(w_{-k})$ given by adding or deleting from the end of $w_{-k}$, and so this ray is also infinite. A similar argument shows that the other ray of source $M(w)$, which is given by adding or deleting from the start of the string, is infinite -- we simply use the same arguments with the string $w^{-1}$.
		
		$(\Leftarrow:)$ Suppose $x$ is an exceptional edge. If $x$ is truncated, then $P(x)$ is uniserial and $S(x)\in \mathcal{M}$, $\rad P(x) \in \mathcal{M}$. Hence, it follows trivially from Lemma~\ref{tubeMouth} that $S(x)$ and $\rad P(x)$ belong to an exceptional tube. So suppose instead that $x$ is non-truncated. Let $v$ be a vertex incident to $x$ and consider the maximal exceptional subtree $T$ (Remark~\ref{SubSubtree}) with connecting vertex $v$. We choose $v$ such that $T$ does not contain $x$.
		
		To show that $S(x)$ belongs to an exceptional tube, we first note that $\mathfrak{e}_v=1$ (since $x$ is exceptional). Let $w=\beta_n^{-1}\ldots\beta_1^{-1}$ be the maximal inverse string such that $\widehat{s}(\beta_n^{-1})=x^v$. Then $e(\beta_i^{-1})$ belongs to $T$ and $e(\beta_i^{-1}) \neq x$ for all $i$. Hence by Lemma~\ref{MaxExceptional}(a), there exists a ray
		\begin{equation*}
			M(w_0) \rightarrow M(w_1) \rightarrow \cdots \rightarrow M(w_k) \rightarrow \cdots
		\end{equation*}
		in $_s\Gamma_A$ such that $w_0=\varepsilon_x$ and $w_k=w$. But $M(w_0) = S(x)$ and $M(w_k) \in \mathcal{M}$. Thus, $M(w_k)$ sits at the mouth of an exceptional tube by Lemma~\ref{tubeMouth} and hence, $S(x)$ belongs to an exceptional tube.
		
		To show that $\rad P(x)$ also belongs to an exceptional tube, let $w'=\gamma_1\ldots\gamma_m$ be the maximal direct string such that $\widehat{e}(\gamma_n)=\overline{x^v}$. Then a similar argument to that used above for $S(x)$ shows that there exists a ray
		\begin{equation*}
			M(w_0) \rightarrow M(w_1) \rightarrow \cdots \rightarrow M(w_k) \rightarrow \cdots
		\end{equation*}
		in $_s\Gamma_A$ such that $w_0=w'$ and $w_k=w' w$. Since $M(w_0) \in \mathcal{M}$ and $M(w_k) = \rad P(x)$, we conclude that $\rad P(x)$ belongs to an exceptional tube.
	\end{proof}
	
	The above theorem as stated shows that the modules $S(x)$ and $\rad P(x)$ for an exceptional edge $x$ belong to exceptional tubes of $_s\Gamma_A$. However, $S(x)$ and $\rad P(x)$ may not necessarily belong to the same exceptional tube. This is due to the construction in the latter part of the proof, where we show that there exist modules $M_1,M_2 \in \mathcal{M}$ with $M_1 \neq M_2$, such that $S(x)$ and $M_1$ belong to the same tube and $\rad P(x)$ and $M_2$ belong to the same tube. We cannot however, guarantee that $M_1$ and $M_2$ belong to the same tube. However, we can use Theorem~\ref{tubeWalks} to describe when $S(x)$ and $\rad P(x)$ do belong to the same exceptional tube of $_s\Gamma_A$.
	
	\begin{cor}
		Given an exceptional edge $\xymatrix@1{u \ar@{-}[r]^x & v}$ in a Brauer graph, $S(x)$ and $\rad P(x)$ belong to the same exceptional tube if and only if $x^u$ and $x^v$ occur within the same double-stepped Green walk.
	\end{cor}
	\begin{proof}
		Let $u,v$ be the vertices connected to $x$. Let $w$ be the string such that $M(w)=\rad P(x)$. Then $w=w'w''$, where $w'$ is the maximal direct with last symbol $\alpha$ such that $\widehat{e}(\alpha)=x^u$ and $w''$ is the maximal inverse string with first symbol $\beta$ such that $\widehat{s}(\alpha)=x^v$. Let $M_1=M(w'')$ and $M_2=M(w')$. Then by the construction in the proof of Theorem~\ref{EdgeTubes}, $S(x)$ and $M_1$ belong to the same tube and $\rad P(x)$ and $M_2$ belong to the same tube. In the case where $x$ is truncated, we actually have $S(x)=M_1$ and $\rad P(x)=M_2$
		
		Recall from Section~\ref{ExcepTubesSec} that to each half-edge $y^t$ in $G$, we have a corresponding string module $M=M(w''') \in \mathcal{M}$. If $t$ is a truncated vertex, then $w'''=\varepsilon_y$. Otherwise, $w'''$ is a direct string with first symbol $\gamma$ such that $\widehat{s}(\gamma)=y^t$. Also recall from \cite[Remark 3.6]{Roggenkamp} that if the $i$-th step along a Green walk from $\overline{y^t}$ is a half-edge $\overline{y_i^{t_i}}$, then $\Omega^i(M) \in \mathcal{M}$ corresponds to the half-edge $y_i^{t_i}$.
		
		Let $\alpha'$ be the first symbol of $w'$ and $\beta'$ be the last symbol of $w''$. Then $\widehat{e}(\beta')$ and $\widehat{s}(\alpha')$ correspond to the modules $M_1$ and $M_2$ respectively. Let $M'_1$ and $M'_2$ be the modules in $\mathcal{M}$ corresponding to the half-edges $x^u$ and $x^v$ respectively. Since $\overline{\widehat{e}(\beta')}$ is the first step along a Green walk from $\overline{x^u}=x^v$ and $\overline{\widehat{s}(\alpha')}$ is the first step along a Green walk from $\overline{x^v}=x^u$, we have $\Omega(M'_1)=M_1$ and $\Omega(M'_2)=M_2$. Note that $M_1$ belongs to the same tube as $M_2$ if and only if $M_1 = \tau^i M_2 = \Omega^{2i} M_2$ for some $i$. This is possible if and only if $M'_1 = \tau^i M'_2$ for some $i$. Since $M'_1$ and $M'_2$ correspond to $x^u$ and $x^v$ respectively, it follows that $x^u$ and $x^v$ belong to the same double-stepped Green walk, as required.
	\end{proof}
	
	\begin{thm} 
		Let $A$ be a representation-infinite Brauer graph algebra and let $x$ and $y$ be (not necessarily distinct) non-exceptional edges of the associated Brauer graph $G$. Then $S(x)$ and $\rad P(y)$ belong to the same component of $_s\Gamma_A$ if and only if either $A$ is 1-domestic or there exists a (not necessarily simple) path 
		\begin{equation*}
			p: \xymatrix@1{u_0 \ar@{-}[r]^{x_1} & v_1 \ar@{-}[r]^{x_2} & v_2 \ar@{-}[r] & \cdots \ar@{-}[r] & v_{n-2} \ar@{-}[r]^-{x_{n-1}} & v_{n-1} \ar@{-}[r]^-{x_n} & u_1}
		\end{equation*}
		of even length in $G$ consisting of non-exceptional edges such that
		\begin{enumerate}[label=(\roman*)]
			\item $x_1=x$ and $x_n=y$;
			\item every edge $x_i$ is not a loop;
			\item if $x_i \neq x_{i+1}$ then $\mathfrak{e}_{v_i}=1$;
			\item if $x_i = x_{i+1}$ then $\mathfrak{e}_{v_i}=2$;
			\item $x_i$ and $x_{i+1}$ are the only non-exceptional edges incident to $v_i$ in $G$.
		\end{enumerate}
	\end{thm}
	\begin{proof}
		($\Leftarrow$:) Suppose $A$ is 1-domestic. Then there exists precisely one component of $_s\Gamma_A$ that is not a tube. Thus by Theorem~\ref{EdgeTubes}, the simple modules and the radicals of the projectives associated to the non-exceptional edges of $G$ all belong to the same component of $_s\Gamma_A$. So suppose instead $A$ is not 1-domestic and a path $p$ satisfying the properties (i)-(v) exists. We will show that $S(x_i)$ is in the same component of $_s\Gamma_A$ as $\rad P(x_{i+1})$ and $S(x_{i+1})$ is in the same component of $_s\Gamma_A$ as $\rad P(x_i)$ for all $i$. Since $p$ is of even length, it will follow from this that $S(x)$ is in the same component as $\rad P(y)$.
		
		There are two possible cases, which arise from conditions (iii) and (iv) respectively. The proof for both cases is similar. In either case, if there are no exceptional edges incident to $v_i$, then $S(x_i)$ is a direct summand of $\rad P(x_{i+1}) / \soc P(x_{i+1})$. Thus, there exists an irreducible morpism $S(x_i) \rightarrow \rad P(x_{i+1})$ and so $S(x_i)$ and $\rad P(x_{i+1})$ belong to the same component of $_s\Gamma_A$. Suppose instead there are exceptional edges incident to $v_i$ and let $\mathfrak{C}_{v_i,\gamma_1}=\gamma_1\ldots\gamma_r\delta_1\ldots\delta_t$, where $s(\gamma_1)=x_i$ and $s(\delta_1)=x_{i+1}$. If $x_i = x_{i+1}$ and $\mathfrak{e}_{v_i}=2$ (condition (iv)) then $t=r$ and $\delta_1=\gamma_1,\ldots,\delta_t=\gamma_r$. Note that (in both cases) the edges $s(\delta_2),\ldots,s(\delta_t)$ and $e(\gamma_1),\ldots,e(\gamma_{r-1})$ each belong to an exceptional subtree with connecting vertex $v_i$ and that $x_i \neq s(\delta_2),e(\gamma_{r-1})$. By Lemma~\ref{MaxExceptional}(a), there exists a ray
		\begin{equation*}
			R: M(w_0) \rightarrow M(w_1) \rightarrow \cdots \rightarrow M(w_k) \rightarrow \cdots
		\end{equation*}
		in $_s\Gamma_A$ such that $w_0 = \varepsilon_{x_i}$ and $w_k=\delta^{-1}_t\ldots \delta^{-1}_2$. Since $(w_k)^{-1}\gamma_1\ldots\gamma_{r-1}$ is a string, it follows from Lemma~\ref{MaxExceptional}(b) that there exists a ray
		\begin{equation*}
			R': \cdots \rightarrow M(w'_{-l}) \rightarrow \cdots \rightarrow M(w'_{-1}) \rightarrow M(w'_0)
		\end{equation*}
		in $_s\Gamma_A$ (which is perpendicular to $R$) such that $w'_0=(w_k)^{-1}$ and $w'_{-l}=(w_k)^{-1}\gamma_1\ldots\gamma_{r-1}$. Now $M(w'_{-l})$ is direct summand of $\rad P(x_{i+1}) / \soc P(x_{i+1})$, and thus, there exists an irreducible morphism $\rad P(x_{i+1}) \rightarrow M(w'_{-l})$. Hence, $S(x_i)=M(w_0)$ is in the same component of $_s\Gamma_A$ as $\rad P(x_{i+1})$, as required. A similar argument shows $S(x_{i+1})$ is in the same component of $_s\Gamma_A$ as $\rad P(x_i)$. Since $n$ is even and $x_1=x$ and $x_n=y$, this implies $S(x)$ and $\rad P(y)$ lie in the same component of $_s\Gamma_A$.
		
		($\Rightarrow$:) Suppose $S(x)$ and $\rad P(y)$ lie in the same component of $_s\Gamma_A$ and consider the strings $w_0$ and $w'_0$ such that $M(w_0)=\rad P(y)$ and $M(w'_0)=S(x)$. Let $L$ be the line in $_s\Gamma_A$ through $M(w_0)$ given by adding or deleting from the end of $w_0$ and let $L'$ be the line in $_s\Gamma_A$ through $M(w'_0)$ given by adding or deleting from the start of $w'_0$. Then $L$ is perpendicular to $L'$ and there exists a module $M(w)$ along $L$ and a module $M(w')$ along $L'$ such that $M(w)=M(w')$. Note that this trivially implies that $w=w'$. We aim to show that this implies that either there exists a path $p$ of the form given in the theorem statement, or $G$ contains precisely one non-exceptional edge, which is a loop (and hence, $A$ is 1-domestic).
		
		First suppose the module at the intersection point of $L$ and $L'$ lies along the ray of target $M(w_0)$. Note that $w_0$ is of the form $\gamma_1\ldots\gamma_r\delta^{-1}_1\ldots\delta^{-1}_t$, where $\gamma_1\ldots\gamma_r$ is maximal direct and $\delta^{-1}_1\ldots\delta^{-1}_t$ is maximal inverse. Let $d$ be the greatest integer such that $e(\delta^{-1}_d)$ is non-exceptional and let $w_{-k}$ be the string $\gamma_1\ldots\gamma_r\delta^{-1}_1\ldots\delta^{-1}_d$. Then it follows from Lemma~\ref{MaxExceptional}(a) that there exists a ray
		\begin{equation*}
			M(w_{-k}) \rightarrow M(w_{-k+1}) \rightarrow \cdots \rightarrow M(w_0) \rightarrow \cdots 
		\end{equation*}
		in $_s\Gamma_A$ such that $|w_{-i}|>|w_{-k}|>0$ for all $0\leq i < k$. Suppose $\widehat{s}(\delta^{-1}_1)=y^u$. Then it follows that $e(\delta^{-1}_d)$ is the edge associated to the first step along a non-exceptional Green walk from $y^u$. Using Lemma~\ref{PathWalk}(b) iteratively starting with the string $w_{-k}$ and the half-edge $\overline{\widehat{e}(\delta^{-1}_d)}$, we conclude that if $M(w_{-i})$ is a module along the ray of target $M(w_0)$ such that $w_{-i}$ ends at a non-exceptional edge $z$, then $z$ belongs to the cycle $c$ of $G$ constructed by performing a non-exceptional Green walk from $y^u$ (described in Remark~\ref{NonExceptionalCycle}). In particular, since $L$ and $L'$ intersect at a module along the ray of target $M(w_0)$, there exists a module $M(w_{-i})$ along this ray such that $e(w_{-i})=x$. Thus, $x$ belongs to $c$. So perform a non-exceptional Green walk from $y^u$ to construct a path
		\begin{equation} \tag{$\dagger$} \label{yPath}
			q: \xymatrix@1{u_0 \ar@{-}[r]^-{y_1} & u_1 \ar@{-}[r] & \cdots \ar@{-}[r] & u_{m-2} \ar@{-}[r]^-{y_{m-1}} &  u_{m-1} \ar@{-}[r]^-{y_m} & u_m}
		\end{equation}
		of even length in $G$, where $y_m=y$, $u=u_{m-1}$ and $y_1=x$. Also note that $\widehat{e}(\delta^{-1}_d)=y_{m-1}^{u_{m-1}}$.
		
		Use Lemma~\ref{PathWalk}(b) iteratively along $q$ starting with $w_{-k}$ until we obtain a ray
		\begin{equation*}
			M(w_{-l}) \rightarrow \cdots \rightarrow M(w_{-k-1}) \rightarrow M(w_{-k}) \rightarrow \cdots \rightarrow M(w_0) 
		\end{equation*}
		along $L$ such that $e(w_{-l})=y_1=x$. Note that it follows from Lemma~\ref{PathWalk}(b) that $|w_{-i}|>|w_{-k}|>0$ for all $i>k$. It also follows that $w_{-l}$ is of the form
		\begin{equation*}
			w_{-l}= \gamma_1\ldots\gamma_r w^+_{m-1} w^-_{m-2}w^+_{m-3}\ldots w^-_2 w^+_1,
		\end{equation*}
		where $w^-_i$ is the direct string of shortest length with first symbol $\alpha_i$ and last symbol $\beta_i$ such that $\widehat{s}(\alpha_i)=y_{i+1}^{u_i}$ and $\widehat{e}(\beta_i)=y_i^{u_i}$, and $w^+_i$ is the inverse string of greatest length with first symbol $\zeta^{-1}_i$ and last symbol $\eta^{-1}_i$ such that $\widehat{s}(\zeta^{-1}_i)=y_{i+1}^{u_i}$ and $\widehat{e}(\eta^{-1}_i)=y_i^{u_i}$. Also note that further use of Lemma~\ref{PathWalk}(b) implies that $w_{-l}$ is a prefix to any string further along the ray of target $M(w_0)$.
		
		Now consider the zero string $w'_0=\varepsilon_x$. We aim to construct the string $w_{-l}$ by adding to the start of $w'_0$, and hence locate $M(w_{-l})$ along the line $L'$. We note that since $w'_0$ is a zero string and since the last symbol $\beta^{-1}$ of $w_{-l}$ is a formal inverse, $w_{-l}$ can only be constructed by adding cohooks to the start of $w'_0$, and hence, $w_{-l}$ lies along the ray of target $M(w'_0)$. In particular, the last symbol of $w'_{-1}={_c(w'_0)}$ must be $\beta^{-1}$.
		
		Suppose $\widehat{e}(\beta^{-1})=x^{v_1}$. Let $x_1=x$ and label the $(i-1)$-th step along a non-exceptional Green walk from $x_1^{v_1}$ by $x_i^{v_i}$. The $i$-th use of Lemma~\ref{PathWalk}(b) on $w'_0$ produces a string $w'_{-k_i}$ along the ray of target $M(w'_0)$ of the form
		\begin{equation*}
			w'_{-k_i}=w'^+_{2i}w'^-_{2i-1} \ldots w'^+_4 w'^-_3 w'^+_2 w'^-_1,
		\end{equation*}
		where $w'^-_j$ is the inverse string of shortest length with first symbol $\zeta'^{-1}_j$ and last symbol $\eta'^{-1}_j$ such that $\widehat{s}(\zeta'^{-1}_j)=x_j^{v_{j-1}}$ and $\widehat{e}(\eta'^{-1}_j)=x_j^{v_j}$, and $w'^+_j$ is the direct string of greatest length with first symbol $\alpha'_j$ and last symbol $\beta'_j$ such that $\widehat{s}(\alpha'_j)=x_{j+1}^{v_j}$ and $\widehat{e}(\beta'_j)=x_j^{v_j}$. Moreover, $w_{-k_i}$ is a suffix of any string $w'$ further along the ray of target $M(w'_0)$.
		
		Since the module $M(w_{-l})$ exists along the ray of target $M(w'_0)$, we may conclude that $w'^-_j=w^+_j$ and $w'^+_j=w^-_j$. Thus, $\widehat{s}(\alpha'_j)=\widehat{s}(\alpha_j)$, $\widehat{s}(\zeta'^{-1}_j)=\widehat{s}(\zeta^{-1}_j)$, $\widehat{e}(\beta'_j)=\widehat{e}(\beta_j)$ and $\widehat{e}(\eta'^{-1}_j)=\widehat{e}(\eta^{-1}_j)$. Hence, $x_{j+1}^{v_j}=y_{j+1}^{u_j}$ and $x_j^{v_j}=y_j^{u_{j}}$. Since each $x_j^{v_j}$ is a step along a non-exceptional Green walk from $x_{j-1}^{v_{j-1}}$ and each $y_{j-1}^{u_{j-1}}$ is a step along a non-exceptional Green walk from $y_j^{u_j}$, this implies that $x_j^{v_j}$ is the first non-exceptional successor to $x_{j-1}^{v_{j-1}}$ and $x_{j-1}^{v_{j-1}}$ is the first non-exceptional successor to $x_j^{v_j}$. This is possible only if there is no half-edge $z^{u_i}$ incident to any $u_i$ in $q$ such that $z$ is non-exceptional and $z^{u_i}\neq y_i^{u_i}, y_{i+1}^{u_i}$. Moreover, $w'^-_j=w^+_j$ and $w'^+_j=w^-_j$ only if $\mathfrak{e}_{v_j} = 1$ if $x_j^{v_j} \neq x_{j+1}^{v_j}$ or $\mathfrak{e}_{v_j} = 2$ if $x_j^{v_j} = x_{j+1}^{v_j}$. Thus, either $q$ is a path of the form $p$ in the theorem statement, or $q$ is a path of length 2 along a loop $x$ in $G$ with incident vertex $v$ such that $\mathfrak{e}_{v} = 1$ and $x$ is the only non-exceptional edge of $G$. In the latter case, $A$ is 1-domestic, as required.
		
		Next suppose the module at the intersection point of $L$ and $L'$ instead lies along the ray of source $M(w_0)=\rad P(y)$ (along $L$). Again let $w_0=\gamma_1\ldots\gamma_r\delta^{-1}_1\ldots\delta^{-1}_t$. Note that $w_0$ ends on a peak, so let $w_1=(w_0)_{-c}=\gamma_1\ldots\gamma_{r-1}$. Now let $d$ be the greatest integer such that $s(\gamma_d)$ is non-exceptional and let $w_k=\gamma_1\ldots\gamma_{d-1}$ if $d>1$ or let $w_k=\varepsilon_{s(\gamma_1)}$ if $d=1$. Then by Lemma~\ref{MaxExceptional}(b), there exists a ray in $_s\Gamma_A$ of the form 
		\begin{equation*}
			\cdots \rightarrow M(w_0) \rightarrow M(w_1) \rightarrow \cdots \rightarrow M(w_{k-1}) \rightarrow M(w_k).
		\end{equation*}
		
		In the case where $w_k=\varepsilon_x$, we note that $x$ and $y$ share a common vertex $v$. Moreover, it follows from the maximality of the string $\gamma_1\ldots\gamma_r$ that $x^v$ and $y^v$ are the only half-edges incident to $v$ such that their corresponding edges are non-exceptional, and that $\mathfrak{e}_{v} = 1$ if $x^{v} \neq y^{v}$ or $\mathfrak{e}_{v} = 2$ if $x^{v} = y^{v}$. Thus in this special case, $x$ and $y$ belong to a path of length 2 that is either of the form $p$ in the theorem statement, or is either a path along a loop $x$ in $G$ with incident vertex $v$ such that $\mathfrak{e}_{v} = 1$ and $x$ is the only non-exceptional edge of $G$.
		
		Otherwise, we note that $e(w_k)$ is non-exceptional and is associated to the first step along a non-exceptional clockwise Green walk from $\widehat{e}(\gamma_r)$. The proof for the case of the intersection point of $L$ and $L'$ belonging to the ray of source $M(w_0)$ is then similar to the proof of the of the case where the intersection point belongs to the ray of target $M(w_0)$. We will summarise the argument. We first construct a path $q$ of the form (\ref{yPath}) above by performing a clockwise non-exceptional Green walk from $\widehat{e}(\gamma_r)$, which we label by $y_m^{u_{m-1}}$. We then use Lemma~\ref{PathWalk}(a) iteratively to produce a string $w_l$ such that $e(w_l)=x$ and $M(w_l)$ is at the intersection of $L$ and $L'$. It follows that $w_l$ is of the form.
		\begin{equation*}
			w_l=\gamma_1\ldots\gamma_{d-1} w^-_{m-2}w^+_{m-3}\ldots w^-_2 w^+_1,
		\end{equation*}
		where $w^-_i$ is the inverse string of shortest length between $y_{i+1}^{u_i}$ and $y_i^{u_i}$ and $w^+_i$ is the direct string of greatest length between $y_{i+1}^{u_i}$ and $y_i^{u_i}$. Since the last symbol of $w_l$ is an arrow, $w_l$ lies in the ray (along $L'$) of source $M(w'_0)$. Let $x_1=x$ and label the $(i-1)$-th step along a non-exceptional clockwise Green walk from $x_1^{v_1}$ by $x_i^{v_i}$. The $i$-th use of Lemma~\ref{PathWalk}(a) on $w'_0$ produces a string
		\begin{equation*}
			w'^+_{2i}w'^-_{2i-1}\ldots w'^+_2 w'^-_1,
		\end{equation*}
		where $w'^-_j$ is the inverse string of shortest length between $x_{j+1}^{v_j}$ and $x_j^{v_j}$ and $w'^+_j$ is the direct string of greatest length between $x_{j+1}^{v_j}$ and $x_j^{v_j}$. So $w'^-_j=w^+_j$ and $w'^+_j=w^-_j$. By similar arguments as presented earlier in the proof, this implies that $q$ is a path of the form $p$ in the theorem statement, or $A$ is 1-domestic.
	\end{proof}

\end{document}